\documentclass[11pt]{amsart}
\usepackage{amssymb}
\usepackage{amscd}
\usepackage{amsmath}
\usepackage{amsmath}  
\usepackage[all]{xy}
\usepackage{mathrsfs}
\usepackage{graphicx}
 \usepackage{color}

\theoremstyle{plain}
\newtheorem{theorem}{Theorem}[section]
\newtheorem{lemma}[theorem]{Lemma}
\newtheorem{corollary}[theorem]{Corollary}
\newtheorem{proposition}[theorem]{Proposition}
\newtheorem{definition}[theorem]{Definition}

\newtheorem{example}[theorem]{Example}

\newtheorem{remark}[theorem]{Remark}

  1

\newcommand{\eins}{\boldsymbol{1}}

\DeclareMathOperator{\Gal}{Gal}

\DeclareMathOperator{\Hom}{Hom}

\DeclareMathOperator{\im}{im}

\newcommand{\CC}{\mathbb{C}}
\newcommand{\bc}{\mathbb{C}}

\newcommand{\GG}{\mathbb{G}}
\newcommand{\NN}{\mathbb{N}}
\newcommand{\QQ}{\mathbb{Q}}

\newcommand{\RR}{\mathbb{R}}
\newcommand{\br}{\mathbb{R}}
\newcommand{\ZZ}{\mathbb{Z}}
\newcommand{\Z}{\mathbb{Z}}

\newcommand{\calF}{\mathcal{F}}
\newcommand{\calL}{\mathcal{L}}

\newcommand{\calP}{\mathcal{P}}

\newcommand{\co}{\mathcal{O}}

\newcommand{\A}{\mathcal{A}}

\newcommand{\W}{\mathcal{W}}

\newcommand{\frp}{\mathfrak{p}}

\newcommand{\Fit}{\mathrm{Fit}}

\newcommand{\bz}{\mathbb{Z}}
\newcommand{\La}{\Lambda}

\newcommand{\D}{\mathrm{d}} 

\newcommand{\G}{\widehat{G}}

\begin{document}

\title[]{On Weil-Stark elements, II: \\
refined Stark conjectures}

\author{David Burns, Daniel Macias Castillo and Soogil Seo}

\begin{abstract} The theory of Weil-Stark elements is used to develop an axiomatic approach to the formulation of refined versions of Stark's Conjecture. This gives concrete new results concerning leading terms of Artin $L$-series and arithmetic properties of Stark elements.\end{abstract}

\address{King's College London,
Department of Mathematics,
London WC2R 2LS,
U.K.}
\email{david.burns@kcl.ac.uk}

\address{Departamento de Matem\'aticas, 
Universidad Aut\'onoma de Madrid, 28049 Madrid (Spain);
and Instituto de Ciencias Matem\'aticas, 28049 Madrid (Spain).}
\email{daniel.macias@icmat.es}

\address{Yonsei University, Department of Mathematics, Seoul, Korea.}
\email{sgseo@yonsei.ac.kr}


\maketitle
%
\section{Introduction}

\subsection{}\label{ga section} We study refined versions of  Stark's seminal conjectures on the algebraic properties of the values at zero of Artin $L$-series. To set the context, we fix a finite Galois extension of global fields $L/K$ of group $G$ and a finite non-empty set of places $S$ of $K$ that contains all archimedean places and all ramifying in $L$. We denote the set of all irreducible complex characters of $G$ by $\widehat{G}$ and for $\chi$ in $\widehat{G}$ we write $L_S(\chi,s)$ for the $S$-truncated Artin $L$-series of $\chi$ and, for each integer $m$, set $L_{S}^{m}(\chi,s) := s^{-m}L_{S}(\chi,s)$. Then, for each non-negative integer $a$, we define an element of the complex group ring $\CC[G]$ of $G$ by setting   
\begin{equation*}\label{a stick def}\theta^{\langle a\rangle}_{L/K,S}(0) := {\sum}_{g \in G}\left({\sum}_\chi\frac{\chi(1)}{|G|}\check\chi(g)L_{S}^{a\cdot\chi(1)}(\check\chi,0)\right)g\in \CC[G]\end{equation*}
where $\check\chi$ is the contragredient of $\chi$ and the inner sum runs over $\chi$ in $\widehat{G}$ for which $L_{S}^{a\cdot\chi(1)}(\check\chi,s)$ is holomorphic at $s=0$. The element $\theta^{\langle 0\rangle}_{L/K,S}(0)$ is the `non-abelian Stickelberger element' of Hayes \cite{hayes} and belongs to the centre $\zeta(\QQ[G])$ of $\QQ[G]$. In general, $\theta^{\langle a\rangle}_{L/K,S}(0)$ belongs to $\zeta(\RR[G])$ and Stark's conjecture, as interpreted by Tate \cite{tate}, predicts it should be equal, up to an undetermined factor in $\zeta(\QQ[G])$, to a regulator constructed from algebraic units. 

To normalise the regulator we write $S_L$ for the set of places of $L$ that lie above a place in $S$, $\mathcal{O}_{L,S}^\times$ for the subring of $L^\times$ comprising all elements that are units at all places outside $S_L$ and $X_{L,S}$ for the subgroup of the free abelian group on $S_L$ comprising elements whose coefficients sum to $0$. We use the Dirichlet regulator isomorphism of $\br[G]$-modules  
\begin{equation}\label{lambdadef} R_{L,S} : \br \otimes_\ZZ \mathcal{O}^\times_{L,S} \rightarrow \br\otimes_\ZZ X_{L,S},\quad 1\otimes u \mapsto -{\sum}_{w\in S_L}\mathrm{log}|u|_w \cdot
w,\end{equation}
where in the sum $|\cdot |_w$ denotes the absolute value at
$w$ (normalised as in \cite[Chap. 0, 0.2]{tate}). Finally, to each homomorphism $\varphi$ in $\Hom_{\RR[G]}(\RR\otimes_\ZZ \mathcal{O}_{L,S}^{\times},\RR\otimes_\ZZ X_{L,S})$, we associate an `($a$-th derived) Artin $\mathscr{L}$-invariant' by setting  
\begin{equation}\label{key elements} \mathscr{L}^a(\varphi)  := \theta^{\langle a\rangle }_{L/K,S}(0)\cdot {\rm Nrd}_{\RR[G]}(\varphi\circ (R_{L,S})^{-1})\in \zeta(\RR[G]),\end{equation}
\noindent{}where ${\rm Nrd}_{\RR[G]}$ denotes the reduced norm of the semisimple algebra $\RR[G]$. 
%

These invariants are of interest for several reasons. For example, in certain cases they encode properties of  canonical special elements (such as Brumer-Stark elements, Rubin-Stark elements or Chinburg-Stark elements) and so are related to problems such as generalized versions of Hilbert's 12th Problem (cf. Tate \cite{tate}), constructive versions of the Deligne-Serre Theorem (cf. Stark \cite{stark77, stark} and Chinburg \cite{chin}) and the construction of new Euler systems (cf. Rubin  \cite{R}). In another more concrete direction, a bound on the `denominator' of $\mathscr{L}^a(\varphi)$ has advantages for the purpose of obtaining evidence for Stark's original conjecture via computer calculations   (cf. the discussion of Dummit in \cite[\S14]{dumm}).

Our main aim then is to investigate the properties of Artin $\mathscr{L}$-invariants and the family of Weil-Stark elements (from \cite{bms1}) will play a key role in our approach. We recall that these elements  extend the classical theory of cyclotomic elements to the setting of general Galois extensions of global fields, and have well-understood arithmetic properties. Here we shall use them to develop an axiomatic approach to the derivation of arithmetic properties of Artin $\mathscr{L}$-invariants, and thereby refined versions of Stark's Conjecture, from the leading term conjecture for Artin $L$-series studied in \cite{dals}. 
\subsection{}\label{concrete} When combined with earlier work (of multiple authors)
, this approach has a range of concrete consequences. To give an idea, we now present two such results, the first focussing on $\mathscr{L}$-invariants and the second on special elements. 

For a non-archimedean place $v$ of $K$ we write ${\rm N}v$ for its absolute norm and fix a place $w_v$ of $L$ above $v$, with decomposition subgroup $G_v$ in $G$ and Frobenius automorphism $\sigma_v$ in $G_v$. We then fix an auxiliary finite set $T$ of places of $K$, disjoint from $S$ and satisfying a mild technical hypotheses (as in (\ref{data fix}) below), and set  
\begin{equation*}\label{delta T def} \gamma_T := {\prod}_{v \in T}{\rm Nrd}_{\QQ[G]}(1-\sigma_{v}^{-1}{\rm N}v) \in \zeta(\QQ[G])^\times.\end{equation*}
We write $\mathcal{S}_S^T(L)$ for the integral dual Selmer group of $\mathbb{G}_m$ over $L$ (cf. \S\ref{selmer section}), 
 $\iota_\#$ for the $\QQ$-linear anti-involution of $\QQ[G]$ that inverts elements of $G$ and use the theory of non-commutative Fitting invariants (as reviewed in \S\ref{app higher fits2}). A precise version of the following result is derived from more general results in \S\ref{cm intro deduction} and \S\ref{gffs}.

\begin{theorem}\label{end result intro} 
Assume $L/K$ and a fixed prime $p$ satisfy either of the following conditions: 
\begin{itemize}
\item[(i)] $L$ and $K$ are function fields and $p$ is any prime;
\item[(ii)] $K$ is totally real, $L$ is CM, $p$ is odd and the Sylow $p$-subgroups of $G$ are abelian. 
\end{itemize}
Fix an integer $a$ with $0\le a < |S|-1$. Then there is an explicit idempotent $e_a$ of $\zeta(\QQ[G])$ that depends only on $a$ and $\{G_v\}_{v \in S}$ and has the following property: if $t$ is the smallest non-negative integer with $p^te_a\in \ZZ_{(p)}[G]$, then for all $\varphi \in \Hom_G(\mathcal{O}_{L,S}^{\times},X_{L,S})$ one has
\begin{equation}\label{bs intro} p^te_a\cdot\mathscr{L}^a(|G|\varphi)\in \gamma_T^{-1}\cdot\iota_\#(\ZZ_{(p)}\otimes_\ZZ{\rm Fit}^{{\rm tr},a}_{\ZZ[G]}(\mathcal{S}_S^T(L))).\end{equation}
\end{theorem}

The idempotent $e_a$ in this result rarely vanishes (and its vanishing also strongly restricts $\mathcal{S}_S^T(L)$). For instance, one has $p^te_0\cdot\mathscr{L}^0(|G|\varphi) = \theta^{\langle 0\rangle}_{L/K,S}(0)$ for all $\varphi$. 
This fact implies, in (i), that (\ref{bs intro}) refines, and generalises, a result of Deligne on divisor class groups of curves  and, in (ii), that (\ref{bs intro}) implies the validity of a refined version of the `non-abelian Brumer-Stark Conjecture' formulated by Nickel \cite{nickel-aif} and  the first author \cite{dals} (cf. Remarks \ref{dals improve} and \ref{deligne rem})

The second result recorded here concerns special elements studied by Stark \cite{stark} and Chinburg \cite{chin} and is proved in \S\ref{cs}. It refers to the Strong-Stark Conjecture, as formulated by Chinburg \cite{chinburgomega} and proved for representations with $\QQ$-valued characters by Tate \cite[Ch. II, Th. 6.8]{tate} and for totally odd characters by Nickel \cite{NickelSS}.

\begin{theorem}\label{lastone2} Assume $K$ is a number field and $|S| > 1$. Let $\co$ be a finite extension of $\ZZ$ and $\psi:G\to{\rm GL}_2(\co)$ an irreducible representation of $G$ of character $\chi_\psi$ that validates the Strong Stark Conjecture and has $L_{S}(\chi_\psi,0) = 0 \not= L^1_S(\chi_\psi,0)$. Then there is a unique place $v_1$ in $S$ with $H^0(G_{v_1},\co^2)\not= (0)$. Assume $w_{v_1}$ is archimedean, 
corresponding to $\sigma: L\to \bc$ and identify $L$ with $\sigma(L)$.  
 Then, writing $\psi_\ast(\gamma_T)$ for the $\chi_\psi$-component of $\gamma_T$, the element 
\[ \varepsilon_{S,T,\psi} := {\rm exp}\bigl(-\psi_\ast(\gamma_T)\cdot L^1_{S}(\chi_\psi,0)\bigr)\]
is a real unit in $L$ that has all of the following properties.
\begin{itemize}
\item[(i)] Either $\varepsilon_{S,T,\psi}$ or $-\varepsilon_{S,T,\psi}$ is
congruent to $1$ modulo all places of $L$ above those in $T$.
\item[(ii)] $|\varepsilon_{S,T,\psi}|_w = 1$ if $w$ is any place of $K$
 that does not lie above $v_1$.
\item[(iii)] For every $g \in G$ one has $-\log |g^{-1}(\varepsilon_{S,T,\psi})|_{w_{v_1}} = (\chi_\psi(g) + \chi_\psi(g\tau))\psi_\ast(\gamma_T)L_{S}^1(\chi_\psi,0)$ where $\tau$ is the unique non-trivial element of $G_{v_1}$.
\item[(iv)] For every map $\phi: \mathcal{O}_L^\times \to \ZZ[G]$ of $G$-modules, the element $2^{-2}|G|^3\phi(\varepsilon_{S,T,\psi})$ belongs to $\ZZ[G]$ and annihilates the $T$-modified ideal class group of $L$.
\end{itemize}
\end{theorem}

Whilst the Chinburg-Stark Conjecture \cite[Conj. 1]{chin} predicts, under the same hypotheses as above, that $\varepsilon_{S,T,\psi}$ has the properties (i), (ii) and (iii),  
 the link to class group structures in claim (iv) is new. 
  In particular, 
   Theorem \ref{lastone2} now strengthens \cite[Th. 4.3.1(ii) and Prop. 12.2.1]{dals} and answers the question raised in \cite[Rem. 12.2.2]{dals}.

\subsection{}\label{Organisation} The main contents of this article are as follows. In \S\ref{n-c alg section} we review relevant results of \cite{bses} and then, in \S\ref{fitt ideal section}, we use these results to prove a canonical decomposition result for the higher non-commutative Fitting invariants of Selmer groups (see Theorem \ref{main result alg}). In \S\ref{hnase section} we extend the classical notion of `Stark element' and explain the connection between these generalised Stark elements, families of special elements in the literature and Artin $\mathscr{L}$-invariants. In \S\ref{5} we review the central conjecture of \cite{dals} and show it implies that all elements considered in \S\ref{hnase section} are `Weil-Stark elements' in the sense of \cite{bms1}. In \S\ref{integrality section} we combine this observation with results from both \S\ref{fitt ideal section} and \cite{bms1} to derive detailed arithmetic properties of Artin $\mathscr{L}$-invariants from the central conjecture of \cite{dals}. These results generalise (to non-abelian extensions and a wider range of Fitting invariants) the results on Selmer groups obtained by Kurihara, Sano and the first author \cite{bks}  and also suggest refinements and generalisations of a range of other results and conjectures (of multiple authors) in the literature (cf. Remarks \ref{dals improve4}, \ref{dals improve cor}, \ref{dals improve} and \ref{dals improve3}). Finally, in \S\ref{last section}, we derive a selection of explicit consequences of our results and, in particular, prove Theorems \ref{end result intro} and \ref{lastone2}.  

\subsection{Acknowledgements} It is a pleasure for the first author to thank Cornelius Greither, Dick Gross, Masato Kurihara, Alice Livingstone Boomla, Cristian Popescu and Takamichi Sano for encouragement and helpful discussions at various stages of this project. In addition, we are grateful to Alexandre Daoud, Henri Johnston, Andreas Nickel and, especially, an anonymous referee for very helpful comments on earlier versions of this article. 

The first author acknowledges funding from Yonsei University in support of a visit in 2019 which benefited this project. The second author acknowledges support for this article as part of Grants CEX2019-000904-S, PID2019-108936GB-C21 and PID2022-142024NB-I00 funded by MCIN/AEI/ 10.13039/501100011033. 
He also thanks the Isaac Newton Institute, and the organisers of the programme `$K$-theory, algebraic cycles and motivic homotopy theory', for support that helped advance this project.

\section{Non-commutative algebra}\label{n-c alg section}


We write $\zeta(R)$ for the centre of a ring $R$. By an $R$-module we mean a left $R$-module unless explicitly stated otherwise. We write $\mathbb{N}$ for the set of natural numbers and set $\mathbb{N}_0 := \mathbb{N}\cup \{0\}$. For $a \in \mathbb{N}$ we write $[a]$ for the set of integers $i$ with $1\le i\le a$.

\subsection{Whitehead orders}\label{2.1} Fix a Dedekind domain $\Lambda$ that is a subring of $\CC$ and has field of fractions $\mathcal{F}$. Let $\A$ be a $\Lambda$-order in a finite-dimensional separable $\mathcal{F}$-algebra $A$. 

\subsubsection{}For each prime ideal $\frp$ of $\Lambda$ we write $\Lambda_{(\frp)}$ for the localisation of $\Lambda$ at $\frp$. For each $\A$-module $M$ and each $\frp$ we then set $M_{(\frp)}:=\Lambda_{(\frp)}\otimes_\Lambda M$, endowed with natural actions of the algebra $\A_{(\frp)}=\Lambda_{(\frp)}\otimes_\Lambda \A$. 
A finitely generated $\A$-module $M$ will be said to be `locally-free' if, for all prime ideals $\frp$ of $\Lambda$, the $\A_{(\frp)}$-module $M_{(\frp)}$ is free. 

The rank ${\rm rk}_\A (M)$ of any such $M$ is a well-defined invariant (independent of the prime $\frp$ with respect to which it is computed). 
We also recall that a finitely-generated locally-free $\A$-module is always projective (see \cite[Prop. (8.19), Vol. I]{curtisr}). The converse of this statement also holds, by a fundamental result of Swan \cite{swan}, if $\A=\Lambda[G]$ for a finite group $G$ for which no prime divisor of $|G|$ is invertible in $\Lambda$ (cf. \cite[Th. (32.11), Vol. I]{curtisr}).

\subsubsection{}\label{Whitehead section}The `Whitehead order' of $\A$ is defined in \cite[Def. 3.1]{bses} via the intersection 
\begin{equation*}\label{Whiteheaddef}\xi(\A):={{{\bigcap}}}_{\frp\in {\rm Spec}(\Lambda)} \xi(\A_{(\frp)}),\end{equation*}
where $\xi(\A_{(\frp)})$ is the $\Lambda_{(\frp)}$-submodule of $\zeta(A)$ generated by $\{ {\rm Nrd}_A(M) : M \in \bigcup_{n\in\NN}{\rm M}_n(\A_{(\frp)})\}$, with ${\rm Nrd}_A$ the reduced norm over $A$. Then $\xi(\A)$ is a $\Lambda$-order in $\zeta(A)$,  $\xi(\A)=\A$ if and only if $\A$ is commutative, and  any surjective map of $\Lambda$-orders $\A\to\mathcal{B}$ induces, upon restriction, a surjective map $\xi(\A)\to\xi(\mathcal{B})$ (cf. \cite[Lem. 3.2]{bses}). In general, $\xi(\A)$ is neither contained in, nor contains, $\zeta(\A)$ but there exists an explicit ideal $\delta(\A)$ of $\zeta(\A)$ with $\delta(\A)\cdot \xi(\A) \subseteq \zeta(\A)$ (cf. \cite[Def. 3.6]{bses}). The links between the order $\xi(\A)$ and ideal $\delta(\A)$ and slightly different constructions of Johnston and Nickel \cite{JN} are discussed in \cite[Rem. 3.3 and 3.8]{bses}.

\subsection{Fitting invariants}\label{app higher fits2}
We fix a finite group $G$ and write $\Lambda$ for $\ZZ$ or $\ZZ_{(p)}$ for some $p$. 
\subsubsection{}
Let $M$ be a matrix in ${\rm M}_{d\times d'}(\QQ[G])$ with $d \ge d'$. Let $\{b_i\}_{i\in[d]}$ be the standard basis of $\Lambda[G]^d$. Then for any integer $t \in \{0\}\cup [d']$ and any $\varphi = (\varphi_i)_{1\le i\le t}$ in $\Hom_G(\Lambda[G]^{d},\Lambda[G])^t$ we write ${\rm Min}^{d'}_{\varphi}(M)$ for the set of all $d'\times d'$ minors of the matrices $M(J,\varphi)$ that are obtained from $M$ by choosing any $t$-tuple of integers $J = \{i_1,i_2,\cdots , i_t\}$ with $1\le i_1< i_2< \cdots < i_t\le d'$, and setting
\begin{equation}\label{fitting matrix} M(J,\varphi)_{ij} := \begin{cases} \varphi_{a}(b_i), &\text{if $j = i_a$ for $a \in [t]$}\\
                            M_{ij}, &\text{otherwise.}\end{cases}\end{equation}
For $a \in \mathbb{N}_0$ the `$a$-th (non-commutative) Fitting invariant of $M$' is the $\xi(\Lambda[G])$-ideal
\[ {\rm Fit}_{\Lambda[G]}^a(M) := \xi(\Lambda[G])\cdot \{{\rm Nrd}_{\QQ[G]}(N): N\in {\rm Min}^{d'}_{\varphi}(M), \, \varphi\in \Hom_{\Lambda[G]}(\Lambda[G]^{d},\Lambda[G])^t, \, t \le a\}.\]
%

\subsubsection{}Fix a finitely generated $\Lambda[G]$-module $Z$. Then a `(finite) free presentation' $h$ of $Z$ is an exact sequence of $\Lambda[G]$-modules 
 $\Lambda[G]^d \xrightarrow{\theta} \Lambda[G]^{d'} \xrightarrow{} Z \to 0$ for natural numbers $d$ and $d'$. The $a$-th Fitting invariant ${\rm Fit}_{\Lambda[G]}^a(h)$ of $h$ is  ${\rm Fit}_{\Lambda[G]}^a(M_\theta)$ with $M_\theta$ the matrix of $\theta$ with respect to the standard bases of $\Lambda[G]^d$ and $\Lambda[G]^{d'}$.

%
%

A `presentation' $h$ of $Z$ is an exact sequence of finitely generated $\Lambda[G]$-modules
\begin{equation}\label{pres seq} P' \xrightarrow{\theta} P \xrightarrow{} Z \to 0;\end{equation}
$h$ is said to be `locally-free', resp. `locally-quadratic', if $P$ and $P'$ are locally-free, resp. locally-free with ${\rm rk}_{\Lambda[G]}(P')={\rm rk}_{\Lambda[G]}(P)$. 
  If $\Lambda = \ZZ$ and $h$ is locally-free, then its localisation $h_{(p)}$ at every prime $p$ is a free presentation of the $\ZZ_{(p)}[G]$-module $Z_{(p)}$, and its $a$-th Fitting invariant is the $\xi(\ZZ[G])$-ideal  
\[ {\rm Fit}_{\ZZ[G]}^a(h) := {\bigcap}_{p}{\rm Fit}_{\ZZ_{(p)}[G]}^a(h_{(p)}).\]
%
%

%
%
The `$a$-th Fitting invariant' of a finitely generated $\Lambda[G]$-module $Z$ is the $\xi(\Lambda[G])$-ideal
\[ {\rm Fit}_{\Lambda[G]}^a(Z) := {{\sum}}_h {\rm Fit}^a_{\Lambda[G]}(h),\]
where in the sum $h$ runs over all locally-free presentations of finitely generated $\Lambda[G]$-modules $Z'$ for which there exists a surjective homomorphism of $\Lambda[G]$-modules $Z' \to Z$. 

\subsubsection{}\label{transpose section} An alternative theory of Fitting invariants is obtained if one replaces the matrices $M(J,\varphi)$ in (\ref{fitting matrix}) by $(M^{\rm tr})(J,\varphi)^{\rm tr}$. This version involves substituting rows (rather than columns) of $M$ by elements of $\im(\varphi_a)$. We write ${\rm Fit}_{\Lambda[G]}^{{\rm tr},a}(h)$ and ${\rm Fit}_{\Lambda[G]}^{{\rm tr},a}(Z)$ for the resulting ideals.  
To relate the two theories one defines the `transpose' $h^{\rm tr}$ of (\ref{pres seq}) to be the presentation 
$$\Hom_\Lambda(P,\Lambda) \xrightarrow{\Hom_\Lambda(\theta,\Lambda)} \Hom_\Lambda(P',\Lambda) \xrightarrow{} {\rm cok}(\Hom_\Lambda(\theta,\Lambda)) \to 0$$
of ${\rm cok}(\Hom_\Lambda(\theta,\Lambda))$ (with the linear duals endowed with contragredient action of $G$); this presentation is locally-free, respectively locally-quadratic, if and only if $h$ is locally-free, respectively locally-quadratic. Then, if $h$ is locally-quadratic, for every $a\in \mathbb{N}_0$, one has 
\begin{equation}\label{transpose equation}{\rm Fit}^{{\rm tr},a}_{\ZZ[G]}(h^{\rm tr}) = \iota_\#\!\bigl({\rm Fit}^a_{\ZZ[G]}(h)\bigr)\end{equation}
(cf. \cite[Lem. 3.24]{bses}). Again, $\iota_\#$ is the anti-involution of $\QQ[G]$ that inverts elements of $G$.

Finally we recall some general facts (from \cite[Th. 3.20]{bses}). If $G$ is abelian (so that $\xi(R[\Gamma]) = R[\Gamma])$, 
then ${\rm Fit}_{\Lambda[G]}^a(Z) = {\rm Fit}_{\Lambda[G]}^{{\rm tr},a}(Z)$ is equal to the classical $a$-th Fitting ideal of $Z$; $({\rm Fit}_{\Lambda[G]}^a(Z))_a$ and $({\rm Fit}_{\Lambda[G]}^{{\rm tr},a}(Z))_a$ are increasing sequences such that for all large enough $a$ 
\[ {\rm Fit}_{\Lambda[G]}^a(Z) = {\rm Fit}_{\Lambda[G]}^{{\rm tr},a}(Z) = \xi(\Lambda[G]);\]
the ideal $\delta(\Lambda[G])\!\cdot\!{\rm Fit}_{\Lambda[G]}^0(Z)= 
\delta(\Lambda[G])\!\cdot\!{\rm Fit}_{\Lambda[G]}^{{\rm tr},0}(Z)$ is contained in $\Lambda[G]$ and annihilates $Z$. 

\section{Fitting invariants of Selmer groups}\label{fitt ideal section}

We prove a new decomposition result for the higher Fitting invariants of Selmer groups. 
\subsection{Selmer groups}\label{selmer section} We recall basic facts concerning the Selmer groups of $\mathbb{G}_m$.

We henceforth fix a finite Galois extension $L/K$ of global fields and set
$G:= \Gal(L/K)$. We write $S^\infty_K$ for the set of archimedean places of $K$ (so $S_K^\infty\not=\emptyset$ if and only if $K$ is a number field) and $S_{L/K}^{\rm ram}$ for the set of places of $K$ that ramify in $L$. For any finite set of places $\Sigma$ of $K$ 
 we write $\Sigma_L$ for the
 set of places of $L$ lying above those in $\Sigma$, $Y_{L,\Sigma}$ for the free abelian group on the set $\Sigma_L$ and $X_{L,\Sigma}$ for the submodule of $Y_{L,\Sigma}$ comprising elements whose coefficients sum to zero. If $E$ is a subfield of $L$, we write $w_E$ for the restriction to $E$ of a place $w$ of $L$. If $S^\infty_K \subseteq \Sigma$, we write $\co_{L,\Sigma}$ for the subring of $L$ comprising
elements integral at all places outside $\Sigma_L$ and $\co_{L,\Sigma}^\times$ for the unit group of $\co_{L,\Sigma}$. We observe that the groups $X_{L,\Sigma}, Y_{L,\Sigma}, \co_{L,\Sigma}$ and $\co_{L,\Sigma}^\times$ are all stable under the natural action of $G$. 
%

We henceforth fix the following data: 
\begin{equation}\label{data fix}
\begin{cases}  &\text{a finite non-empty set of places $S$ of $K$ with $S^\infty_K\cup S^{\rm ram}_{L/K} \subseteq S$};\\
               &\text{a finite set $T$ of places of $K$ with $S\cap T=\emptyset$ and $\co^\times_{L,S,T}$ torsion-free,}\end{cases}
\end{equation}
where $\co_{L,S,T}^\times := \mathcal{O}_{L,S}^\times\cap L_T^\times$, with $L_T^\times := \{ a\in L^\times : {\rm ord}_w(a-1)>0 \text{ for all } w\in T_L \}$. 

We write ${\rm Cl}_S^T(L)$ for the quotient of the group of fractional $\mathcal{O}_L$-ideals
whose support is disjoint from $(S\cup T)_L$ by the subgroup
of principal ideals with a generator congruent to $1$ modulo all
places in $T_{L}$. We observe that ${\rm Cl}_S^T(L)$ is a $G$-module extension of ${\rm Cl}_S(L)$. 

The `($S$-relative $T$-trivialised) integral dual Selmer group for $\GG_m$ over $L$' is defined in \cite[Def. 2.1]{bks} (where it is denoted $\mathcal{S}_{S,T}(\GG_m/L)$) to be the cokernel $\mathcal{S}_S^T(L)$ of the injective map  
\begin{equation}\label{definingSelmer}
 {\prod}_{w \notin (S\cup T)_L}\ZZ \stackrel{\theta_{S,T}}{\longrightarrow} \Hom_\ZZ(L_T^\times,\ZZ) , 
\end{equation}
where $\theta_{S,T}$ sends $(x_w)_w$ to the map $(a \mapsto {\sum}_{w\notin S_L\cup T_L}{\rm ord}_w(a)x_w)$.

In the sequel we shall use the following facts concerning this module (all of which are proved in \cite[\S2]{bks}). 
There exists a canonical exact sequence
\begin{equation}\label{selmer lemma I} 0 \to {\rm Cl}_{S}^T(L)^\vee \to \mathcal{S}_S^T(L) \to \Hom_{\ZZ}(\mathcal{O}^\times_{L,S,T},\ZZ)\to 0\end{equation}
and a canonical transpose $\mathcal{S}_S^T(L)^{\rm tr}$ to $\mathcal{S}_S^T(L)$ (in the sense of Jannsen's homotopy theory of modules \cite{jannsen}) that lies in a canonical exact sequence
\begin{equation}\label{selmer lemma II} 0 \longrightarrow {\rm Cl}_{S}^T(L) \longrightarrow
\mathcal{S}_S^T(L)^{\rm tr}
\longrightarrow X_{L,S} \longrightarrow 0.\end{equation}
In (\ref{selmer lemma I}) and throughout the sequel, for any $G$-module $M$ we set $M^\vee:=\Hom_\ZZ(M,\QQ/\ZZ)$, endowed with the natural contragredient action of $G$. 

We finally recall that for each normal subgroup $H$ of $G$ and $E=L^H$, there exists a canonical isomorphism of $G/H$-modules
\begin{equation}\label{Seldescent}\ZZ[G/H]\otimes_{\ZZ[G]}\mathcal{S}_S^T(L)^{\rm tr} \cong \mathcal{S}_S^T(L^H)^{\rm tr}.\end{equation}

In the sequel, if $T= \emptyset$, then we freely omit it from all of the above notations. Similarly, if $K$ is a number field and $S = S^\infty_K$, then we will often omit it from all notations.
%

\subsection{Characters and idempotents}

%
 We abbreviate to ${\bf 1}_G$ or even ${\bf 1}$ the trivial character of $G$ and write $\check\psi$ for the contragredient of a character $\psi$. The primitive central idempotent of $\CC[G]$ at $\psi$ is $e_\psi:=(\psi(1)/|G|){\sum}_{g\in G}\psi(g)g^{-1}$. For a normal subgroup $H$ of $G$ we set $T_H := {\sum}_{h \in H}h \in \ZZ[G]$ and we write $e_H$ for the central idempotent $|H|^{-1}\cdot T_H$ of $\QQ[G]$.


For $a \in \mathbb{N}_0$, we then write $\widehat G'_{a,S}$ for the (possibly empty) subset of $\widehat{G} \setminus \{{\bf 1}\}$ comprising characters $\psi$ that satisfy both of the following conditions
\begin{equation}\label{new def} \begin{cases} &(\text{i})\,\,\,\text{the set} \,\,\, S_\psi:= \{v \in S: \psi(T_{G_v}) = |G_v|\cdot \psi(1)\}\,\,\, \text{has cardinality $a$, and} \\
                 &(\text{ii})\,\,\,\psi(T_{G_v}) = 0\,\,\,\text{ for all } \,\,\, v \in S\setminus S_\psi.\end{cases}\end{equation}
Here $G_v$ denotes the decomposition subgroup in $G$ of any fixed place of $L$ above $v$.
We then define $\widehat G_{a,S}$ to be $\widehat G_{a,S}' \cup \{{\bf 1}\}$ if $a=|S|-1$ and to be $\widehat G_{a,S}'$ if $a\not= |S|-1$.

We also consider the (possibly empty) set of characters 
\begin{equation*}\label{new def2} \widehat G'_{(a),S} := \{\psi \in \widehat{G}\setminus \{{\bf 1}\}: |S_\psi| \ge a \},\end{equation*}
and define $\widehat G_{(a),S}$ to be $\widehat G_{(a),S}' \cup \{{\bf 1}\}$ if $a < |S|$ and to be $\widehat G_{(a),S}'$ if $a \ge |S|$.

We then obtain central idempotents of $\QQ[G]$ by setting
\[ e_{a,S} := {\sum}_{\psi \in \widehat G_{a,S}} e_\psi\,\,\,\,\text{ and }\,\,\,\, e_{(a),S} := {\sum}_{\psi \in \widehat G_{(a),S}}e_\psi. \]
%
We also define a subset of $S$ by setting

\[ S^a_{\rm min} := \begin{cases} \bigcup_{\psi\in \widehat G'_{a,S}} S_\psi, &\text{if $G$ is not trivial,}\\
                                  S\setminus \{v_0\}, &\text{if $G$ is trivial and $a = |S|-1$},\\
                                  \emptyset, &\text{if $G$ is trivial and $a \not= |S|-1$},\end{cases}\]
where $v_0$ is a fixed place in $S$ (the choice of which will not matter in the sequel).

For any set $\Sigma$ and $a\in \mathbb{N}_0$ we write $\wp_a(\Sigma)$ for the set of subsets of $\Sigma$ of cardinality $a$. We abbreviate $\wp_a(S_{\rm min}^a)$ to  $\wp^\ast_a(S)$ and for $v$ in $S$ we set 
\[ \wp_a(S,v) := \{I\in \wp^\ast_a(S): v \notin I\}.\]

Finally, for each $I$ in $\wp^\ast_a(S)$ we define a central idempotent of $\QQ[G]$ by setting
\begin{equation}\label{def eI} e_I := e_{\bf 1} + {\sum}_{\{\psi \in \widehat G'_{a,S}: S_\psi = I\}} e_\psi.\end{equation}
%

\begin{remark}\label{na comms}{\em \

(i) To interpret (\ref{new def})(i) more explicitly, note that for $\psi$ in $\widehat G\setminus \{{\bf 1}\}$ one has $\psi(T_{G_v}) = |G_v|\cdot \psi(1)$ if and only if $G_v$, and hence also its normal closure, is contained in $\ker(\psi)$.

(ii) If $\psi$ in $\widehat G\setminus \{{\bf 1}\}$ is linear, then $\psi\in \widehat G'_{|S_\psi|,S}$ since, in this case, the validity of (\ref{new def})(ii) follows from the definition of $S_\psi$. Thus, if $G$ is abelian, then $\widehat G = \bigcup _{a'\in \mathbb{N}_0}\widehat G_{a',S}$ and there are also non-abelian extensions for which the same is true. For example, if $L/K$ is generalised dihedral, with $A$ the abelian subgroup of index two in $G$, 
and all places in $S$ split in $L^A/K$, then for every $\phi$ in $\widehat A\setminus \{{\bf 1}\}$ one has ${\rm Ind}_A^G(\phi) \in \widehat G'_{a,S}$ with $a = |\{v \in S: G_v \subseteq \ker(\phi)\}|$.}

{\em (iii) Set $V := \{v \in S: G_v \,\text{ is trivial}\}$ and $r_{\rm sp} := |V|$. Then $V \subseteq S_\psi$ for all $\psi$ in $\widehat G\setminus \{{\bf 1}\}$ and so $\widehat G_{a,S}= \emptyset$ if $a< r_{\rm sp}$ and $\widehat G_{(a),S} = \widehat G$ if $a\le r_{\rm sp}$. In particular, if any $\widehat G'_{a,S}$ with $a > 0$ contains a faithful character of $G$, then $a=r_{\rm sp}$, $S^a_{\rm min} = \Sigma$ and $\widehat G = \widehat G_{(a),S}$. (In this regard, we note Schur's Lemma implies no element of $\widehat G$ can be faithful if $G$ has a non-cyclic centre and that a complete classification of groups with  faithful irreducible characters has been given by Gaschutz in \cite{Gaschutz}).}\end{remark}

%

\subsection{The decomposition result}\label{decompo section} 

For each place $v$ in $S$ we write $N_v$ for the normal closure of $G_v$ in $G$ and
 abbreviate the central idempotent $e_{N_v}$ to $e_v$. 

\begin{theorem}\label{main result alg} 
Fix sets $S$ and $T$ as in (\ref{data fix}) and a place $v\in S$. Then the $G$-module $\mathcal{S}^T_S(F)$ has a locally-quadratic presentation $h=h_v$ with the property that,
for every $a\in \mathbb{N}_0$, there is a direct sum decomposition of $\xi(\ZZ[G])$-modules
\begin{equation}\label{decomp} (1-e_v+e_{\bf 1})e_{(a),S}\cdot{\rm Fit}_{\ZZ[G]}^{{\rm tr},a}(h) = c^a_{S,v}e_{\bf 1}\cdot{\rm Fit}_{\ZZ[G]}^{{\rm tr},a}(h)\oplus {\bigoplus}_{I\in \wp_a(S,v)} e_I\cdot{\rm Fit}_{\ZZ[G]}^{{\rm tr},a}(h).\end{equation}
Here we set $c^{a}_{S,v}:=0$ if $a\geq|S|$ and, otherwise, we use the integer 
\[ c^a_{S,v} := |\{v\}\setminus S_{\rm min}^a| - {\rm min}\{|S_{\rm min}^a|, |S\setminus S_{\rm min}^a|\} - \delta_{a0}.\]
The same claim is also true if one replaces $\mathcal{S}^T_S(F)$ and ${\rm Fit}_{\mathbb{Z}[G]}^{{\rm tr},a}(-)$ by $\mathcal{S}^T_S(F)^{\rm tr}$ and 
${\rm Fit}_{\mathbb{Z}[G]}^{a}(-)$. 
\end{theorem}

\begin{remark}{\em If $G$ is abelian, then this result recovers the main algebraic result (Theorem 1.1) of Livingstone Boomla and the first author \cite{dbalb}. In addition, an earlier version of the result for non-abelian $G$ was first obtained in unpublished work of the same authors and we are grateful to Livingstone Boomla for her permission to include it here. }\end{remark}

The proof of Theorem \ref{main result alg} will occupy the rest of this section. Throughout the argument, for 
%
%
%
%
%
%
$a \in \mathbb{N}_0$ we set 
\[ \mathcal{F}_a := {\rm Fit}_{\ZZ[G]}^a(\mathcal{S}^T_S(L)^{\rm tr}).\]

\subsubsection{} We first record a result that will allow us to treat a special case of Theorem \ref{main result alg}. 

\begin{lemma}\label{fit vanishing} The following claims are valid for each $a\in \mathbb{N}_0$.

\begin{itemize}
\item[(i)] For $\psi$ in $\widehat G_{(a),S}\setminus \widehat G_{a,S}$ the space $e_\psi (\CC\otimes_\ZZ \mathcal{F}_a)$ vanishes.
\item[(ii)] If $|S|> a+1$, then $e_{\bf 1}\cdot \mathcal{F}_a$ vanishes.
\item[(iii)] If $v \in S\setminus S^a_{\rm min}$, then $(e_v-e_{\bf 1})e_{(a),S}\cdot\mathcal{F}_a$ vanishes.
\end{itemize}
\end{lemma}

\begin{proof} We claim first that for $\psi$ in $\widehat G$, $a \in \mathbb{N}_0$ and a finitely generated $G$-module $Z$ the sublattice $e_\psi\cdot {\rm Fit}_{\ZZ[G]}^a(Z)$ of $\zeta(\CC[G])$ will vanish whenever $n_{\psi}(Z)> a\cdot\psi(1)$, where we set $n_\psi(Z) := {\rm dim}_{\CC}(e_\psi (\CC\otimes_\ZZ Z))$.

It is enough to prove this after localising at any given prime $p$ and to do this we fix a free presentation $\Pi$ of a finitely generated $\ZZ_{(p)}[G]$-module $Z'$ for which there exists a surjective homomorphism of $\ZZ_{(p)}[G]$-modules $Z' \to Z$.
Let $M_\Pi$ be a matrix in ${\rm M}_{d\times d'}(\CC[G])$ for integers $d$ and $d'$ with $d\ge d' \ge a$, that represents $\Pi$ and let $M'$ be any matrix that is obtained from $M_\Pi$ by changing at most $a$ of its columns. For any $\psi$ in $\widehat G$ write $M_{\Pi,\psi}$ and
$M'_\psi$ for the matrices in ${\rm M}_{d\cdot\psi(1)\times d'\cdot\psi(1)}(\CC)$ that are respectively obtained by splitting the matrices $e_\psi\cdot  M_\Pi$ and $e_\psi\cdot M'$ in ${\rm M}_{d\times d'}(e_\psi\cdot\CC[G])$.

Then the column rank of $M_{\Pi,\psi}$ is $d'\cdot\psi(1) - n_\psi(Z')$ and so the column rank of
  $M'_\psi$ is at most $d'\cdot \psi(1) - n_{\psi}(Z') + a\cdot\psi(1)$. In particular, if $n_\psi(Z) > a\cdot\psi(1)$, then also $n_{\psi}(Z') > a\cdot\psi(1)$ and so the column rank of $M'_\psi$ is strictly less than $d'\cdot \psi(1)$. This implies that the determinant of any
  $d'\cdot \psi(1)\times d'\cdot \psi(1)$ minor of $M'_\psi$ vanishes and so, since $e_\psi\cdot {\rm Fit}_{\ZZ[G]}^a(Z)$ is generated by the determinants of all such minors (see the definition in \cite[\S 3]{bses}), it must also vanish, as claimed. Using this observation, and the fact that for any non-trivial $\psi$ in $\widehat G_{(a),S}\setminus \widehat G_{a,S}$ there exists a place $v$ in $S\setminus S_\psi$ for which $H^0(G_v,V_{\psi})$ does not vanish, one derives claim (i) by noting the exact sequence (\ref{selmer lemma II}) implies that 
\[
 n_\psi(\mathcal{S}^T_S(L)^{\rm tr}) = \begin{cases} {\sum}_{v \in S}{\rm dim}_\CC(H^0(G_v,V_\psi)) > {\sum}_{v \in S_\psi}{\rm dim}_\CC(H^0(G_v,V_\psi))\\
 \hskip 1.65truein = {\sum}_{v \in S_\psi}{\rm dim}_\CC(V_\psi) = a\cdot \psi(1), &\text{if $\psi \not= {\bf 1}$,}\\
 |S|-1 > a = a\cdot\psi(1), &\text{if $\psi = {\bf 1}$,}\end{cases}\]
where the second equality in the case $\psi \not= {\bf 1}$ follows from Remark \ref{na comms}(i).

Claim (ii) is then an immediate consequence of claim (i) in the case $\psi = {\bf 1}$. 
 To derive claim (iii) from claim (i) it suffices to prove that if $\psi$ belongs to $\widehat {G}_{a,S}\setminus \{{\bf 1}\}$, then $N_v$ cannot be contained in $\ker(\psi)$ (and hence $e_\psi(e_v-e_{\bf 1}) = 0$). This follows from the fact that, as $v$ does not belong to $S^a_{\rm min}$, the inclusion $N_v\subseteq \ker(\psi)$ would imply that $|S_\psi| \ge 1 + a$.
\end{proof}

\subsubsection{}\label{special cases} We now fix a place $v=v_0$ in $S$ and a place $w_0$ of $L$ above $v_0$. We set $n:= |S|-1$, label (and thereby order) the elements of $S$ as
$\{v_0\}\cup \{ v_i \}_{i \in [n]}$ and write $S_0$ for $S\setminus \{v_0\}$.

For $i\in [n]$ we now fix a place $w_i$ of $L$ above $v_i$ and for each intermediate field $E$ of $L/K$ write $w_{i,E}$ for the place of $E$ obtained by restriction of $w_i$ (so $w_{i,L} = w_i$).
We consider the canonical composite surjective homomorphism $ \varrho_{L,S}: \mathcal{S}_S^T(L)^{\rm tr} \to X_{L,S} \to Y_{L,S_0},$ 
where the first arrow is the canonical homomorphism occurring in (\ref{selmer lemma II}) and the second arrow maps $w-w_0$ to $0$ if $v_0$ divides $w$ and to $w$ otherwise.


Any exact sequence in claim (ii) of \cite[Prop. 5.6]{bms1} (in the case $\Pi=\ZZ[G]$) defines a locally-quadratic presentation of $G$-modules 
\begin{equation*}\label{fixed h}P^0\stackrel{\phi}{\to}P\stackrel{\varpi}{\to}\mathcal{S}_S^T(L)^{\rm tr}\to 0\end{equation*}
in which $P$ is free of finite rank $d>n+1$. In fact, a slightly more general version of claim (iii) of \cite[Prop. 5.6]{bms1} (see \cite[\S 5.4]{bks} or \cite[Lem. 6.15]{bses2}) gives a presentation $\tilde h=\tilde h_{v}$ as above together with an ordered $G$-basis $\widehat{\underline{b}}=\{\hat b_i\}_{i \in [d]}$ of $P$ for which
\begin{equation}\label{extrabasisproperties}\varrho_{L,S}(\varpi(\hat b_i))=\begin{cases}w_{i},\,\,\,\,\,\,\,\,\,\,\text{if }i\in[n],\\ 0,\,\,\,\,\,\,\,\,\,\,\,\,\,\,\text{if }i\in[d]\setminus[n].\end{cases}\end{equation}

\begin{remark}\label{asinLY}{\em The presentation $\tilde h$ constructed above has two useful properties. \

\noindent{}(i) Fix a normal subgroup $H$ of $G$. Then, with respect to the identification (\ref{Seldescent}) , the $H$-coinvariance $\tilde h_H$ of $\tilde h$ gives a locally-quadratic presentation of the $\ZZ[G/H]$-module $\mathcal{S}_S^T(L^H)^{\rm tr}$ that satisfies (\ref{extrabasisproperties}) after replacing $L$ and each $\hat b_i$ and $w_i$ by $L^H$, the image of $\hat b_i$ in $P_H$ and $w_{i,L^H}$ respectively. In particular, with respect to the identification $e_H\cdot\zeta(\QQ[G]) = \zeta(\QQ[G/H])$ one has $ e_H\cdot {\rm Fit}^a_{\ZZ[G]}(\tilde h) = {\rm Fit}^a_{\ZZ[G/H]}(\tilde h_H)$ for all $a \in \mathbb{N}_0$, and the proof of Theorem \ref{main result alg} shows that $\tilde h_H$ satisfies the equality (\ref{decomp}) relative to the extension $L^H/K$.

\noindent{}(ii) The definition of the transpose Selmer group $\mathcal{S}_S^T(L)^{\rm tr}$ in \cite{bks} ensures that the transpose $\tilde h^{\rm tr}$ of $\tilde h$ (in the sense of \S\ref{transpose section}) is a locally-quadratic presentation of $\mathcal{S}_S^T(L)$. This combines with the general equality (\ref{transpose equation}), and the fact that all idempotents occurring in the statement of Theorem \ref{main result alg} are fixed by $\iota_\#$, to reduce the proof of Theorem \ref{main result alg} (with $h = \tilde h^{\rm tr}$) to the proof of (\ref{decomp}) with each term ${\rm Fit}_{\ZZ[G]}^{{\rm tr},a}(h)$ replaced by ${\rm Fit}^a_{\ZZ[G]}(\tilde h)$.  
}\end{remark}

\subsubsection{} We set $\mathcal{F}'_a:={\rm Fit}^a_{\ZZ[G]}(\tilde h)\subseteq\mathcal{F}_a$. By using Lemma \ref{fit vanishing}, and Remark \ref{asinLY}(ii), we can now quickly prove Theorem \ref{main result alg} in several special cases. 

\begin{lemma}\label{second step} Theorem \ref{main result alg} is valid if either $F=k$ or $a=0$.
\end{lemma}

\begin{proof} In the first case $G$ is trivial and $\mathcal{F}'_a$ is the $a$-th Fitting ideal of the abelian group $\mathcal{S}^T_S(L)^{\rm tr}$. The validity of Theorem \ref{main result alg} in this case therefore follows directly from \cite[Prop. 3.3]{dbalb}. In the rest of the argument, we will therefore assume that $a=0$ and $G$ is not trivial.

In this case, Lemma \ref{fit vanishing} (i) and (iii) combine to imply that the left hand side of (\ref{decomp}) is equal to
$e_{0,S}\cdot \mathcal{F}'_0$. In addition, $S^a_{\rm min}=\emptyset$ and so for each $v$ in $S$ one has $c_{S,v}^0 = 1 - 0 - 1 = 0$ and $\wp_0(S,v) = \{\emptyset\}$. If $|S| = 1$, then (\ref{def eI}) implies that 

\[ e_\emptyset = e_{\bf 1} + {\sum}_{\psi \in \widehat G'_{0,S}}e_\psi = {\sum}_{\psi \in \widehat G_{0,S}}e_\psi  = e_{0,S},\] %
whilst if $|S| >1$ then $e_\emptyset = e_{\bf 1} +e_{0,S}$. Since $e_{\bf 1} \cdot \mathcal{F}_0 = 0$ if $|S|>1$ by Lemma \ref{fit vanishing}(ii), this shows that in both cases the right hand side of (\ref{decomp}) is equal to $e_{0,S}\cdot \mathcal{F}'_0$, as required.
\end{proof}

The next result deals with the case that $a$ is `large'. 

\begin{proposition}\label{first step} Theorem \ref{main result alg} is valid if $a \ge |S|-1$. \end{proposition}

\begin{proof} By using the result of Lemma \ref{fit vanishing}(i) one can easily check that both sides of (\ref{decomp}) vanish if $a > |S|$. It is thus enough to consider $a$ equal to either $|S|-1$ or $|S|$. Our argument then splits into several subcases, depending on the cardinality of $S_{\rm min}^a$. Following Lemma \ref{second step} we always assume $a>0$.

We consider first the case that $S_{\rm min}^a = \emptyset$ and $a>0$, and hence that $\widehat G'_{a,S} = \emptyset$. In this case Lemma \ref{fit vanishing}(i) and (ii) combine to imply $e_{(a),S}\cdot \mathcal{F}'_a = e_{a,S}\cdot \mathcal{F}'_a$ is equal to $e_{\bf 1}\cdot \mathcal{F}'_a$ if $a = |S|-1$ and vanishes if $a = |S|$. One also has $\wp_a(S,v) = \emptyset$ (since $a >0$) and so the right hand side of (\ref{decomp}) is equal to $c_{S,v}^a e_{\bf 1}\cdot \mathcal{F}'_a$. The claimed equality is thus true in this case since $c_{S,v}^a$ is equal to $1-0-0 =1$ if $a = |S|-1$ and to $0$ if $a = |S|$.

In the remainder of the argument we assume $S_{\rm min}^a \not= \emptyset$ and hence that $|S_{\rm min}^a|$ is equal to either $|S|-1$ (so $a = |S|-1$) or $|S|$ (so $S_{\rm min}^a = S$).

We consider first the case $|S_{\rm min}^a|=|S|-1 = a$ and $v \in S_{\rm min}^a$. In this case one has $N_v \subseteq \ker(\psi)$ for all $\psi \in \widehat G_{a,S}'$ so $(1-e_v)e_{a,S} = 0$ and Lemma \ref{fit vanishing}(i) 
implies the left hand side of (\ref{decomp}) is equal to $e_{\bf 1}\cdot\mathcal{F}'_a$. Given this, the claimed equality follows from the fact that $\wp_a(S,v) = \emptyset$ and $c_{S,v}^a = 0-1-0 = -1$.

We next assume $|S_{\rm min}^a|=|S|-1 = a$ and $v \notin S_{\rm min}^a$ so that $S_{\rm min}^a = S\setminus \{v\}$. In this case Lemma \ref{fit vanishing}(i) and (iii) together imply that the left hand side of (\ref{decomp}) is equal to $e_{a,S}\cdot\mathcal{F}'_a$. In addition, one has $c_{S,v}^a = 1 - 1-0 = 0$ and the unique element of $\wp_a(S,v)$ is equal to $I:= S_{\rm min}^a$ so $e_I = e_{a,S}$ and the right hand side of (\ref{decomp}) is also equal to $e_{a,S}\cdot\mathcal{F}'_a$.

We now consider the case $S_{\rm min}^a = S$ and $a = |S|-1$. For $v$ in $S$ we write $N(v)$ for the subgroup generated by $N_{v'}$ as $v'$ varies over $S\setminus \{v\}$. Then for each $\psi$ in $\widehat G_{a,S}'$ one has $e_\psi e_v = 0$ if and only if $N(v) \subseteq \ker(\psi)$. This implies $(1-e_v)e_{a,S} = (e_{N(v)}-e_{\bf 1})e_{a,S}$ and hence that the left hand side of (\ref{decomp}) is equal to $e_{N(v)}e_{a,S}\cdot \mathcal{F}'_a$. The claimed equality thus follows from the fact that, in this case, $c_{S,v}^a = 0$, $\wp_a(S,v)$ has a single element $I_v := S\setminus \{v\}$ and it is straightforward to check that $e_{I_v}$ is equal to $e_{N(v)}e_{a,S}$.

Finally, we assume $a = |S|$ and $S^a_{\rm min}\not= \emptyset$ (and hence that both $S_{\rm min}^a = S$ and $G$ is not trivial). In this case one verifies that $e_{(a),S} = e_{a,S} = e_{N_S}-e_{\bf 1}$ with $N_S$ the normal subgroup of $G$ generated by $N_v$ as $v$ varies over $S$. For each $v$ in $S$ one therefore has $(1-e_v+ e_{\bf 1})e_{(a),S} = (1-e_v)e_{N_S} =0$, where the last equality is valid since $N_v\subseteq N_S$, and so the left hand side of (\ref{decomp}) vanishes. On the other hand, in this case the right hand side of (\ref{decomp}) vanishes since for any $v$ in $S$ both $c_{S,v}^a = 0$ and $\wp_a(S,v)$ is empty.

This completes the proof of Theorem \ref{main result alg} in the case that $a \ge |S|-1$. \end{proof}

\subsubsection{}\label{main case} Following Lemma \ref{second step} and Proposition \ref{first step} we assume in the remainder of the argument that $G$ is not trivial and that $0< a < |S|-1$.

We note that in this case Lemma \ref{fit vanishing}(ii) implies that the term $1-e_v+e_{\bf 1}$ on the left hand side of (\ref{decomp}) can be replaced by $1-e_v$ and that the first summand on the right hand side of (\ref{decomp}) can be omitted. 

Now Lemma \ref{fit vanishing}(i) implies that the term $e_{(a),S}$ in the left hand side of (\ref{decomp}) can be replaced by $e_{a,S}$, so we are left to prove that
\begin{equation}\label{enoughdecomp}(1-e_v)e_{a,S}\cdot\mathcal{F}'_a={\bigoplus}_{I\in\wp_a(S,v)}e_I\cdot\mathcal{F}'_a\end{equation}
and, from the definition of non-commutative Fitting invariants, it is enough to prove that this equality is valid after localising at every $p$.
We henceforth fix a prime $p$. 

Then, since both of the $\ZZ_{(p)}[G]$-modules $P^0_{(p)}$ and $P_{(p)}$ are free of rank $d$, after modifying $\phi$ appropriately if necessary, we may and will assume that $P^0_{(p)}=P_{(p)}$. We then write $M_p$ for the matrix of $\phi_{(p)}$ with respect to the induced basis $\widehat{\underline{b}}$ of $P_{(p)}$ as in (\ref{extrabasisproperties}).

For each $J\subseteq [d]$ we write $m_{p}(J)$ for the ideal of $\xi(\ZZ_{(p)}[G])$ generated by reduced norms of elements in the set $M_{p}(J)$ of all matrices obtained from $M_{p}$ by replacing all entries in each column indexed by an integer in $J$ by arbitrary elements of $\ZZ_{(p)}[G]$. Then, by the very definition of $a$-th Fitting invariants, we have that
\begin{equation}\label{key equality} \mathcal{F}'_{a,(p)} = \Fit^a_{\ZZ_{(p)}[G]}(M_p) = {\sum}_{J\in \wp_a(d)} m_{p}(J)\end{equation}
with $\wp_a(d) := \wp_a(\{i \in \ZZ: 1\le i\le d\})$. Next we note that, for $\psi$ in $\widehat G \setminus \{{\bf 1}\}$, one has 
\[ v \notin S_\psi \Leftrightarrow N_v \not\subseteq \ker(\psi) \Leftrightarrow (1-e_v)e_\psi \not=0\]
and hence, 
%
%
 for $a\in \mathbb{N}_0$, also 
\begin{equation}\label{useful idem relation} (1-e_v)e_{a,S} = {\sum}_{I \in\wp_a(S,v)}(e_I-e_{\bf 1}).\end{equation}
This equality combines with (\ref{key equality}) to imply that

\begin{align*} (1-e_v)e_{a,S}\cdot\mathcal{F}'_{a,(p)}
=\, &({\sum}_{I\in \wp_a(S,v)}(e_I-e_{\bf 1}))\cdot({\sum}_{J\in \wp_a(d)}m_{p}(J))\\
=\, &{\sum}_{I \in \wp_a(S,v)}(e_I-e_{\bf 1})\cdot m_{p}(I) \\
=\, &{\sum}_{I\in \wp_a(S,v)}(e_I-e_{\bf 1}) \cdot \mathcal{F}'_{a,(p)},\end{align*}
where the second and third equalities follow from the result of Lemma \ref{lemma one} below.

To deduce the equality (\ref{enoughdecomp}) in the case that $a < |S|-1$ it thus suffices to observe firstly that the last sum in the above formula is direct since for distinct elements $I_1$ and $I_2$ of $\wp^\ast_a(S)$ the idempotents $e_{I_1}-e_{\bf 1}$ and $e_{I_2}-e_{\bf 1}$ are orthogonal, and then that for each such $I$ Lemma \ref{fit vanishing}(ii) implies $(e_I -e_{\bf 1})\cdot \mathcal{F}'_{a,(p)} = e_I \cdot \mathcal{F}'_{a,(p)}$.

In the following result we use the correspondence $v_i \leftrightarrow i$ to identify $\wp_a(S_0)$ with the subset $\wp_a([n])$ of $\wp_a([d])$.

\begin{lemma}\label{lemma one} Take $\psi\in \widehat G_{a,S}\setminus \{{\bf 1}\}$ with $v \notin S_\psi$. Then $S_{\psi}$ belongs to $\wp_a([d])$ and for each $J$ in $\wp_a([d])\setminus \{S_\psi\}$ one has $e_\psi \cdot m_{p}(J) =0$. \end{lemma}

\begin{proof} Set $G_\psi := G/\ker(\psi)$ and $L^\psi := L^{{\rm ker}(\psi)}$. We write $q_\psi:\QQ[G]\to\QQ[G_\psi]$ for the canonical projection.
We  write $M_{p,\psi}$ for the image in $M_d(\ZZ_{(p)}[G_\psi])$ of $M_{p}$ under the map that applies $q_\psi$ to each entry. We claim that each column of $M_{p,\psi}$ that corresponds to an integer in $S_\psi\setminus J$ must vanish. 

In that case, if $J \not= S_\psi$, then the image in $M_d(\ZZ_{(p)}[G_\psi])$ of any matrix in $M_{p}(J)$ would have at least one column of zeroes and so its reduced norm in $\xi(\ZZ_{(p)}[G_\psi])$ would have to vanish. It would follow that $e_\psi \cdot m_{p}(J)=0$, as claimed.

To finally prove our claim we 
fix a place $v_k\in S_\psi$ corresponding to the $k$-th column of $M_{p,\psi}$ with $k$ in $S_\psi\setminus J$. Then $v_k$ splits completely in $L^\psi/K$.

For any $i\in[d]$ we write $\phi_{(p)}(\hat b_i)={\sum}_{j=1}^{j=d}\mu_{i,j}\hat b_j$ with each $\mu_{i,j}$ in $\ZZ_{(p)}[G]$. Then since $\phi_{(p)}(\hat b_i)$ belongs to $\im(\phi_{(p)})=\ker(\varpi_{(p)})$, the properties (\ref{extrabasisproperties}) imply that in $Y_{L,S_0,(p)}$ one has 
\begin{equation*}\label{zerocolumns}0=\varrho_{L,S,(p)}(\varpi_{(p)}(\phi_{(p)}(\hat b_i)))={\sum}_{j=1}^{j=d}\mu_{i,j}\varrho_{L,S,(p)}(\varpi_{(p)}(\hat b_j))={\sum}_{j=1}^{j=n}\mu_{i,j}w_{j,L}.\end{equation*}

Now this equality implies that $\mu_{i,j}w_{j,L}$ vanishes for every $i\in[d]$ and $j\in[n]$. It follows that the element $q_\psi(\mu_{i,k})\cdot w_{k,L^\psi}$ vanishes in $Y_{L^\psi,S_0,(p)}$ for every $i\in[d]$. Since $v_k$ splits completely in $L^\psi/K$, we finally deduce that $q_\psi(\mu_{i,k})=0$ for every $i\in[d]$, as was claimed.
\end{proof}


\section{Generalised Stark elements}\label{hnase section}

In this section we use the theory of reduced exterior powers, as developed by Sano and the first author in \cite{bses}, to formulate, in Definition \ref{hnase}, an appropriate generalisation of the classical notion of `Stark element'.  

To do this, we fix a Dedekind domain $\Lambda$ that is a subring of $\CC$ and write $\mathcal{F}$ for its field of fractions. We let $\A$ be a $\Lambda$-order in a finite-dimensional separable $\mathcal{F}$-algebra $A$ and set $A_\CC:=\CC\otimes_{\mathcal{F}}A$. Given an $\A$-module $M$ we write $M_{\rm tor}$ for its $\Lambda$-torsion submodule and $M_{\rm tf}$ for the quotient of $M$ by $M_{\rm tor}$. 
For any field extension $E$ of $\mathcal{F}$ we also set $E\cdot M:=E\otimes_{\Lambda}M$.

Throughout the section we also fix a (non-empty) set of places $S$ of $K$ as in (\ref{data fix}).

\subsection{Bimodules and functors}\label{dataasin}

We assume to be given a finitely generated $(\mathcal{A},\Lambda[G])$-bimodule $\Pi$ that has both of the following properties. 


\begin{itemize}
\item[($\Pi_1$)] $\Pi$ is a locally-free $\mathcal{A}$-module.
 \item[($\Pi_2$)] The association $W\mapsto W\otimes_\mathcal{A}\Pi$ induces an injection from the set of isomorphism classes of simple right $A_\CC$-modules to the set of isomorphism classes of simple right $\CC[G]$-modules.
\end{itemize}


\begin{remark}\label{wedd bij}{\em The property ($\Pi_2$) induces a bijection $\Pi_*$ from the set of Wedderburn components ${\rm Wed}_A$ of $A_\CC$ to a subset $\Upsilon_\Pi$ of $\widehat G$ and hence a commutative diagram 
\begin{equation}\label{key commute}\begin{CD}
K_1(\QQ[G]) @> {\rm Nrd}_{\QQ[G]}  >> \zeta(\QQ[G])^\times @> \subset  >> \zeta(\CC[G])^\times = {\prod}_{\widehat G}\CC^\times\\
@V \mu^1_\Pi VV @. @VV \iota_\Pi V\\
 K_1(A) @> {\rm Nrd}_A >> \zeta(A)^\times @>\subset  >> \zeta(A_\CC)^\times = {\prod}_{{\rm Wed}_A}\CC^\times.
\end{CD}\end{equation}
Here $K_1(A)$ is the $K_1$-group of the category of finitely generated $A$-modules, $\mu^1_\Pi$ sends each automorphism $\alpha$ of a finitely generated left $\QQ[G]$-module $V$ to the class of the  automorphism ${\rm id}_\Pi\otimes_{\ZZ[G]}\alpha$ of $\Pi\otimes_{\ZZ[G]}V$ and $\iota_\Pi$ sends $(z_\chi)_{\chi\in \widehat G}$ to $(z_{\Pi_*(C)})_{C\in {\rm Wed}_A}$.}\end{remark}

\begin{example}\label{exam1}{\em In concrete applications, the following bimodules $\Pi$ will arise.\

(i) If $\Lambda\subset\QQ$ and $A$ is a direct factor of $\QQ[G]$, then for any homomorphism of rings $\kappa:\Lambda[G] \to \mathcal{A}$  one can set $\Pi_\kappa := \mathcal{A}$. In all cases the lattice $\Pi_\kappa$ has the properties $(\Pi_1)$ and $(\Pi_2)$, and $\Upsilon_{\Pi_\kappa}$ is the subset of $\widehat{G}$ comprising characters which occur in $A_\CC$. 

(ii) As a special case of (i), for any $a\in\NN_0$ with $a < |S|$ and any set $I\in \wp^\ast_a(S)$, we obtain a $\Lambda$-order $\Lambda[G]e_I$ in the algebra $\QQ[G]e_I$. Then the $(\Lambda[G]e_I,\Lambda[G])$-bimodule $\Pi_I:=\Lambda[G]e_I$ has the properties $(\Pi_1)$ and $(\Pi_2)$ with $\Upsilon_{\Pi_I}=\{\eins\}\cup\{\chi\in\G'_{a,S}:\,S_\chi=I\}$. Analogous claims are valid with $e_I$ replaced by $e_{(a),S}$.

(iii) As another special case of (i), fix a cyclic subgroup $C$ of $G$ and an idempotent $e_{-}$ of $\Lambda[C]$ that is central in $\Lambda[G]$ and has the property that $\iota_\#(e_{-})=e_{-}$. As explained in \cite[Exam. 3.2 (iii)]{bms1}, the $\Lambda$-order $\Lambda[G]e_{-}$ is locally-Gorenstein in the sense of Def. 3.1 of loc. cit., and the $(\Lambda[G]e_{-},\Lambda[G])$-bimodule $\Pi_{-}:=\Lambda[G]e_{-}$ has the properties $(\Pi_1)$ and $(\Pi_2)$, as well as the property $(\Pi_3)$ stated in loc. cit.

(iv) Let $\psi$ be a representation $G \to {\rm GL}_{\psi(1)}(\mathcal{O}_\psi)$ for a finite extension $\mathcal{O}_\psi$ of $\La$. 
 Set $\mathcal{A}_\psi := \mathcal{O}_\psi$ and $\Pi_\psi := \mathcal{O}_\psi^{\psi(1)}$, regarded as an $(\mathcal{O}_\psi,\Lambda[G])$-bimodule via $\psi$. Then 
$\Pi_\psi$ has the properties $(\Pi_1)$ and $(\Pi_2)$ 
with $\Upsilon_{\Pi_\psi} = \{\psi\}$.
}\end{example}

\begin{remark}{\em Any bimodule $\Pi$ as above gives rise to functors $M \mapsto {^\Pi}M$ and $M\mapsto {_\Pi}M$ from the category of left $G$-modules to the category of left $\mathcal{A}$-modules by setting
\[ {^\Pi}M := H^0(G,\Pi\otimes_\ZZ M)\,\,\,\text{ and }\,\,\,{_\Pi}M:= H_0(G,\Pi\otimes_\ZZ M) = \Pi\otimes_{\ZZ[G]}M,\]
where the left action of $G$ on the tensor product is via $g(\pi\otimes m) = (\pi)g^{-1}\otimes g(m)$. These functors are respectively left and right exact.}\end{remark}

We next fix a surjective homomorphism of $\mathcal{A}$-modules 
\begin{equation}\label{surj Pi map} \pi: {_{\Pi}}X_{L,S}\to Y_\pi\end{equation}
in which $Y_\pi$ is locally-free. We then define an associated central idempotent of $A$ by setting
 $e_\pi:={\sum} e$, where the sum is over all primitive idempotents $e$ of $\zeta(A)$ with  $e(\mathcal{F}\cdot\ker(\pi))=0$.

\begin{remark}\label{cases of r}{\em  We write $r_{L,S}^\Pi$ for the maximal possible rank of a locally-free $\mathcal{A}$-quotient of the module ${_{\Pi}}X_{L,S}$. Then one has $e_\pi\neq 0 \Longleftrightarrow {\rm rk}_\A(Y_\pi)=r_{L,S}^\Pi$. In addition, in natural cases it is possible to describe $r_{L,S}^\Pi$ explicitly. 
For example, if $(_{\Pi}X_{L,S})_{\rm tf}$ is a locally-free $\mathcal{A}$-module, as is automatically the case if $\mathcal{A}$ is a Dedekind domain, then 
 $r_{L,S}^\Pi = {\rm rk}_{\mathcal{A}}((_{\Pi}X_{L,S})_{\rm tf})$. In general, if $v\in S$, then $_{\Pi}Y_{L,S\setminus\{v\}}$ is a quotient of $_{\Pi}X_{L,S}$ and so $|S|\geq r_{L,S}^\Pi\ge | S^{\Pi}_v|$ with
\[ S^{\Pi}_v := \{ v' \in S\setminus\{v\}: \text{ the }\mathcal{A}\text{-module }\,H_0(G_{v'},\Pi)_{\rm tf} \,\text{ is both non-zero and locally-free}\}.\]
%
Note also that, in the setting of Example \ref{exam1}(i), the $\mathcal{A}$-module $H_0(G_{v'},\Pi_\kappa)_{\rm tf}$ is non-zero and locally-free if $\kappa$ sends each element of $G_{v'}$ to the identity of $A$. 
}\end{remark}

\subsection{Equivariant derivatives of $L$-series}


\subsubsection{}For any $\chi$ in $\G$ 
and any $r\in \mathbb{N}_0$ with $r \leq {\rm ord}_{s=0}L_{S}(\chi,s)$, we set
$$L^r_{S}(\chi,0):=\lim_{s\to 0} s^{-r}L_{S}(\chi,s).$$
We then associate to our fixed map $\pi$ in (\ref{surj Pi map}) an element
of $\zeta(\CC[G])$ by setting

\[ \theta_{S}^{\pi}(0) := {\sum}_{\chi\in \Upsilon_\Pi} e_\chi L_{S}^{r_{\pi}(\chi)}(\check\chi,0)\quad\text{with}\quad r_\pi(\chi) := {\rm dim}_\CC(W_\chi\otimes_{A_\CC}(\CC\!\cdot\! Y_\pi)).\]
%
Here we have fixed, for each $\chi\in \Upsilon_\Pi$, a simple right $A_\CC$-module $W_\chi$ for which the associated $\CC[G]$-module $W_\chi\otimes_{A_\CC}(\CC\cdot \Pi)$ has character $\chi$. 
%
%

In the next result we refer to the non-negative integer $r^\Pi_{L,S}$ defined in Remark \ref{cases of r} and to the map $\iota_\Pi$ defined in Remark \ref{wedd bij}.

\begin{lemma}\label{stick vanishing} $\iota_\Pi(\theta_{S}^{\pi}(0)) = e_\pi\iota_\Pi(\theta_{S}^{\pi}(0))$ and, if ${\rm rk}_\mathcal{A}(Y_\pi) \not= r_{L,S}^\Pi$, then $\theta_{S}^{\pi}(0)=0$.
\end{lemma}

\begin{proof} For each $\chi$ in $\widehat{G}$ set $V_\chi := W_\chi\otimes_{A_\CC}(\CC\cdot \Pi)$. Then, since this module has character $\chi$, an analysis of the functional equation of Artin $L$-series shows that the order of vanishing at $s=0$ of the meromorphic function $L_{S}(\chi,s)$ is given by
\begin{equation}\label{fe equality} {\rm ord}_{s=0}L_{S}(\chi,s) = {\rm dim}_\CC( H^0(G,\Hom_\CC(V_{\check\chi},\CC\!\cdot\! X_{L,S})))\end{equation}
(for details see, for example, \cite[Chap. I, Prop. 3.4]{tate}). Taken together with the natural composite isomorphisms of vector spaces
\begin{align*} H^0(G,\Hom_\CC(V_{\check\chi},\CC\cdot X_{L,S})) \cong &\, H_0(G,\Hom_\CC(V_{\check\chi},\CC\!\cdot\! X_{L,S}))\\
 \cong &\, H_0(G,V_{\chi}\otimes_\CC \CC\!\cdot\! X_{L,S}) \\
 \cong &\,  V_{\chi}\otimes_{\CC[G]}(\CC\!\cdot\! X_{L,S}) \cong W_\chi\otimes_{A_\CC}(\CC\cdot{_\Pi}X_{L,S})\end{align*}
(where the last isomorphism follows from the definition of ${_\Pi}X_{L,S}$) this shows that
\begin{equation*}\label{orderassum} {\rm ord}_{s=0}L_{S}(\chi,s) =  {\rm dim}_\CC(W_\chi\otimes_{A_\CC}(\CC\cdot{_\Pi}X_{L,S})) = r_\pi(\chi) + {\rm dim}_\CC ( W_\chi\otimes_{A_\CC}\ker(\pi)).\end{equation*}

This formula implies, in particular, that if $W_\chi\otimes_{A_\CC}\ker(\pi)\not= (0)$, then $L_{S}^{r_\pi(\chi)}(\chi,0)=0$. This in turn implies the claimed equality
$\iota_\Pi(\theta_{S}^{\pi}(0)) = e_\pi\iota_\Pi(\theta_{S}^{\pi}(0))$ since each character $\chi$ in $\Upsilon_\Pi$ for which $W_\chi\otimes_{A_\CC}\ker(\pi)$ vanishes corresponds (via the projection $\iota_\Pi$ in (\ref{key commute})) to a unique primitive central idempotent $e$ of $A$ for which $e(\QQ\cdot \ker(\pi))$ vanishes.

The second claim also follows from the same argument. This is because if ${\rm rk}_\mathcal{A}(Y_\pi) \not= r_{L,S}^\Pi$, then ${\rm rk}_\mathcal{A}(Y_\pi) < r_{L,S}^\Pi$ so that $\QQ\cdot \ker(\pi)$ contains a submodule isomorphic to $A$ and hence $W_\chi\otimes_{A_\CC}\ker(\pi)$ does not vanish for any $\chi$ in $\Upsilon_\Pi$. 
\end{proof}

\subsubsection{}Upon appropriate specialisation, the above construction recovers elements that have arisen in previous studies of Stark's conjecture. 

To explain this we note that, for $a\in \mathbb{N}_0$, the definition of $\G_{(a),S}$ ensures that the  function 
\[ \theta^{(a)}_{L/K,S}(s):= {\sum}_{\chi\in \G_{(a),S}} e_\chi L_{S}^{a\cdot\chi(1)}(\check\chi,s).\]
is holomorphic at $s=0$. The following result describes the link between the value at $s=0$ of this function, values of the form $\theta_S^\pi(0)$ and the element $\theta^{\langle a\rangle}_{L/K,S}(0)$ used in the Introduction. 



\begin{lemma}\label{exam 2} Fix $a\in \NN_0$ with $a < |S|$ and $I\in \wp^\ast_a(S)$. Write $\Pi$ for the bimodule $\Pi_I$ defined in Example \ref{exam1}(ii), and $\pi_I$ for the natural (surjective) map from ${_\Pi}X_{L,S}$ to ${_\Pi}Y_{L,I}$.  

\begin{itemize}
\item[(i)] The $\Lambda[G]e_I$-module ${_{\Pi}}Y_{L,I}$ is free of rank $a$, and so $\pi_I$ is a map of the form (\ref{surj Pi map}).
\item[(ii)] $\theta^{\pi_I}_{S}(0) = e_I(\theta^{(a)}_{L/K,S}(0))$. 

\item[(iii)] If $a < |S|-1$, then for $v\in S$ one has $\theta^{\langle a\rangle}_{L/K,S}(0)(1-e_v)e_{(a),S} = {\sum}_{I\in \wp_a(S,v)} \theta^{(a)}_{L/K,S}(0)e_I.$
\end{itemize}
\end{lemma}

\begin{proof} There is a canonical decomposition of $\Lambda[G]e_I$-modules
\[ {_{\Pi}}Y_{L,I} := \Lambda[G]e_I\otimes_{\ZZ[G]}Y_{L,I}={\bigoplus}_{v\in I}(\Lambda[G]e_I\otimes_{\ZZ[G_v]}\ZZ).\]
Since by Remark \ref{na comms}(i) one has $G_v\subseteq\ker(\psi)$ for each $v$ in $I$ and each $\psi$ with $S_\psi=I$, as well as for $\psi=\eins$, each summand in the above decomposition is $\Lambda[G]e_I$-free of rank one. This proves claim (i). 

To prove claim (ii) we note $\Upsilon_\Pi=\{\eins\}\cup\{\chi\in\G'_{a,S}:\,S_\chi=I\}$. 
By claim (i), one also has $r_{\pi_I}(\chi)=a\cdot\chi(1)$ for every $\chi\in\Upsilon_\Pi$. Hence,  since $e_I={\sum}_{\chi\in\Upsilon_\Pi}e_\chi$, the claim is true since  
\[ \theta^{\pi_I}_{S}(0) = {\sum}_{\chi\in \Upsilon_\Pi} e_\chi L_{S}^{a\cdot\chi(1)}(\check\chi,0)
=  e_I\bigl({\sum}_{\chi\in\G_{(a),S}} e_\chi L_{S}^{a\cdot\chi(1)}(\check\chi,0)\bigr) = e_I(\theta^{(a)}_{L/K,S}(0)).\]

The equality in claim (iii) is true since $\theta^{(a)}_{L/K,S}(0)e_{\bf 1} = 0$ (as $a < |S|-1$) and hence 
\begin{align*} {\sum}_{I\in \wp_a(S,v)} \theta^{(a)}_{L/K,S}(0)e_I =&\, {\sum}_{I\in \wp_a(S,v)} \theta^{(a)}_{L/K,S}(0)(e_I-e_{\bf 1})\\ 
=&\, \theta^{(a)}_{L/K,S}(0)(1-e_v)e_{a,S} = \theta^{\langle a\rangle}_{L/K,S}(0)(1-e_v)e_{(a),S}.
\end{align*}
Here the second equality follows from (\ref{useful idem relation}) and the third is true since $\theta^{(a)}_{L/K,S}(0)e_\psi = \theta^{\langle a\rangle}_{L/K,S}(0)e_\psi$ for $\psi\in \widehat G_{a,S}$ whilst (\ref{fe equality}) combines with the argument of Lemma \ref{fit vanishing}(i) to imply $\theta^{\langle a\rangle}_{L/K,S}(0)e_\psi = 0$ for $\psi\in\widehat G_{(a),S}\setminus \widehat G_{a,S}$.  
\end{proof}

\begin{example}\label{exam 2r}{\em \

\noindent{}(i)  Write $V$ for the subset of $S$ comprising places that split completely in $L/K$ and assume $V \not= S$. Then $\wp_{|V|}^*(S)=\{V\}$ and ${\rm ord}_{s=0}L_{S}(\chi,s)\geq |V|\cdot\chi(1)$ for $\chi\in\G$. For $\chi\neq\eins$, this is an equality if and only if $S_\chi=V$. It follows that $\theta^{\pi_V}_{S}(0)$ is equal to $e_V(\theta^{(|V|)}_{L/K,S}(0))=\theta^{(|V|)}_{L/K,S}(0)$ and so  %
coincides with the `higher-order Stickelberger element' studied in \cite{dals}. \


 \noindent{}(ii) 
In the setting of Example \ref{exam1}(iv) there is a natural surjection $\pi=\pi^{\psi}$ of $\mathcal{O}_\psi$-modules from 
${_{\Pi_\psi}}X_{L,S}$ to $({_{\Pi_\psi}}X_{L,S})_{\rm tf}$. 
For this map one has $\theta_{S}^{\pi}(0) = e_\psi L_{S}^\ast(\check\psi,0)$, where  $L^\ast_{S}(\check\psi,0)$ denotes the leading term of $L_{S}(\check\psi,s)$ at $s=0$.
}\end{example}

\subsection{The definition of generalised Stark elements}\label{Stark elements} 
%
\subsubsection{}\label{reduced review} For the reader's convenience, we quickly recall the relevant algebraic constructions from \cite{bses}. For this we note that, 
for some finite index set $I$, there is a product decomposition   
$A = {\prod}_{i\in I}A_i$,
in which each component $A_i$ is both simple and unique up to isomorphism. For each index $i$ fix a splitting field $E_i$ for $A_i$ over $\zeta(A_i)$ and a simple $E_i\otimes_{\zeta(A_i)}A_i$-module $V_i$ (cf. Remark \ref{fixedconventions} below regarding these choices).

Then, for each $r \in \mathbb{N}_0$, there exists a canonical `$r$-th reduced exterior power functor' 
$M \mapsto {{\bigwedge}}^r_AM$ 
from the category of finitely generated $A$-modules to the category of $\zeta(A)$-modules. This functor behaves well under scalar extension.   
For any finitely generated $A$-module $M$ and any integer $s$ with $0\leq s\leq r$, there exists a natural duality pairing
$${\bigwedge}_A^rM\times{\bigwedge}_{A^{\rm op}}^s\Hom_A(M,A)\to{\bigwedge}_A^{r-s}M; \quad (m,\varphi)\mapsto\varphi(m).$$

If one fixes ordered $E_i$-bases of the spaces $V_i$, then ordered subsets $\{m_j\}_{1\leq j\leq r}$ of $M$ and $\{\varphi_j\}_{1\leq j\leq r}$ of $\Hom_A(M,A)$ give rise to elements $\wedge_{j=1}^{j=r}m_j$ of ${\bigwedge}_A^rM$ and $\wedge_{j=1}^{j=r}\varphi_j$ of ${\bigwedge}_{A^{\rm op}}^r\Hom_A(M,A)$. It is known that $(\wedge_{i=1}^{i=r}\varphi_i)(\wedge_{j=1}^{j=r}m_j)$ 
belongs to $\zeta(A)$ and 
depends only on the subsets $\{m_j\}_{1\leq j\leq r}$ and $\{\varphi_j\}_{1\leq j\leq r}$ (see \cite[Lem. 4.10]{bses}). 


For a finitely generated (left) $\A$-module $N$, the `$r$-th reduced Rubin lattice' is the $\xi(\A)$-lattice obtained by setting 
$${{{{\bigcap}}}}_\mathcal{A}^rN:=\{a\in{\bigwedge}_A^r(\mathcal{F}\cdot N):(\wedge_{i=1}^{i=r}\varphi_i)(a)\in\xi(\mathcal{A})\text{ for all }\varphi_1,\ldots,\varphi_r\in\Hom_\mathcal{A}(N,\mathcal{A})\}.$$
This lattice is independent of the chosen bases of the $V_i$, and its basic properties are established in \cite[Th. 4.17]{bses}. 

\begin{remark}\label{fixedconventions}{\em Different choices of the modules $\{V_i\}_{i \in I}$ give rise to isomorphic reduced exterior powers (cf. \cite[Rem. 4.4]{bses}). 
In the present setting, all constructions are made over the fixed global field $K$ and will be normalised as follows. We fix an algebraic closure $K^c$ of $K$ and a representation ${\rm Gal}(K^c/K)\to{\rm GL}_{\chi(1)}(\CC)$ for each irreducible complex character of ${\rm Gal}(K^c/K)$ with open kernel. Then this family of representations gives rise to all of the data necessary to construct a functorial theory of reduced exterior powers over algebras of the form $\mathcal{F}[{\rm Gal}(M/K)]$ with $M$ a finite Galois extension of $K$ in $K^c$ (cf. \cite[Rem. 4.9]{bses}). In particular, for $M/K$ is abelian, the above constructions recover classical exterior powers and the `Rubin lattices' introduced in \cite{R} (cf. \cite[Rem. 4.18]{bses}).
}\end{remark}

\subsubsection{}
With $\pi$ as in (\ref{surj Pi map}), we now set $r:={\rm rk}_\A(Y_\pi)$ and fix an $r$-tuple $x_\bullet = (x_i)_{1\le i\le r}$ in $Y_\pi$. Then, by definition of the idempotent $e_\pi$, the Dirichlet regulator isomorphism $R_{L,S}$ from (\ref{lambdadef}) induces an isomorphism of $\zeta(e_\pi A_\CC)$-modules
\[ \lambda^{\pi}_{L,S}: e_\pi({{\bigwedge}}_{A_\CC}^{r} \CC\cdot {^\Pi}\mathcal{O}_{L,S}^\times) \stackrel{\sim}{\to} e_\pi({{\bigwedge}}_{A_\CC}^{r} \CC\cdot Y_{\pi})\]
and, since $\iota_\Pi(\theta_{S}^{\pi}(0)) = e_{\pi}\iota_\Pi(\theta_{S}^{\pi}(0))$ (by Lemma \ref{stick vanishing}), the element $\iota_\Pi(\theta_{S}^{\pi}(0)) \cdot \wedge_{i=1}^{i=r}x_i$ belongs to $\im(\lambda_{L,S}^{\pi})$. We may therefore make the following definition.

\begin{definition}\label{hnase} {\em The `generalised Stark element (relative to $x_\bullet$)' is the element
%
$$\varepsilon^\pi_{x_\bullet} := (\lambda^{\pi}_{L,S})^{-1}(\iota_\Pi(\theta_{S}^{\pi}(0)) \cdot \wedge_{i=1}^{i=r}x_i) \in e_\pi({{\bigwedge}}_{A_\CC}^{r} \CC\cdot {^\Pi}\mathcal{O}_{L,S}^\times).$$
%
}
\end{definition}

The following examples show that this definition extends several different constructions that have been made in the literature.

\begin{remark}\label{rubin-stark context}{\em  
If $a<|S|$ and $I$ belongs to $\wp_a^*(S)$, for each $v$ in $I$ we fix a place $w_v$ of $L$ above $v$. We then define an $a$-tuple $x_{I}=\{w_v\}_{v\in I}$ in $Y_{L,I}$ and obtain an element
\begin{equation}\label{eta def} \varepsilon_{L/K,S}^{I}:=\varepsilon^{\pi_I}_{x_{I}}\in e_I\cdot{\bigwedge}_{\CC[G]}^a (\CC\cdot\mathcal{O}^\times_{L,S}).\end{equation}
%

If $I$ comprises places that split completely in $L/K$, then $\varepsilon_{L/K,S}^I$ 
can be thought of as a `non-abelian Rubin-Stark element'.  
In particular, if $G$ is abelian and $I$ is the set of all places in $S$ that split completely in $L$, then Lemma \ref{exam 2}(ii) and Example \ref{exam 2r}(i) imply $\varepsilon_{L/K,S}^{I}$ is the element $\varepsilon_{S}$ of $\br\cdot{\wedge}^{r}_{\ZZ [G]}\mathcal{O}^\times_{L,S}$ used by Rubin in \cite[Conj. B$'$]{R}.} \end{remark}

\begin{remark}\label{stark context}{\em Assume the notation and hypotheses of Examples \ref{exam1}(iv) and \ref{exam 2r}(ii). Write $\pi$ for the map $\pi^\psi$, set $\mathcal{O}:= \mathcal{O}_\psi$ and write $E$ for the quotient field of $\mathcal{O}$ and $r$ for the rank of the (locally-free) $\mathcal{O}$-module $({_{\Pi_\psi}}X_{L,S})_{\rm tf}$. Let $\underline{b} = \{ b_{i}\}_{1\le i\le r}$ be any subset of $({_{\Pi_\psi}}X_{L,S})_{\rm tf}$ that is linearly independent over $\mathcal{O}$. Then the data $(\underline{b},\pi)$ can be used in Definition \ref{hnase} and, in this case, the Stark elements $\varepsilon^\pi_{\underline{b}}$ are related to those that occur in the refined Stark conjecture of \cite[Conj. 2.6.1]{dals}. In particular, if $r=1$ and $\psi$ is non-trivial, then the formula (\ref{fe equality}) implies that there is a unique place $v$ in $S$ for which $H^0(G_v,\Pi_\psi)$ is non-zero and, for suitable primitive idempotents $f_\psi$ of $E[G]$, the elements $\varepsilon^\pi_{\{|G|\cdot f_\psi (w)\}_{w\mid v}}$ recover the elements that are studied in Stark's original articles \cite{stark0, stark} and the subsequent article of Chinburg \cite{chin}. For more details of these connections see the discussion in \S\ref{cs}.}
\end{remark}

\subsubsection{}\label{L-invariants}

We end this section by clarifying the relationship between generalised Stark elements and the Artin $\mathscr{L}$-invariants defined in (\ref{key elements}). For this purpose, we assume given a second set $T$ as in (\ref{data fix}) and associate to any $\psi \in \Hom_G(\mathcal{O}_{L,S,T}^{\times},X_{L,S})$ a $\zeta(\RR[G])$-valued regulator
\begin{equation}\label{equivariantreg}R(\psi):={\rm Nrd}_{\RR[G]}((\RR\cdot\psi)\circ (R_{L,S})^{-1}).\end{equation}


For each $a \in \mathbb{N}_0$ with $a < |S|$ and each $I\in \wp^\ast_a(S)$ we 
consider the maps
\[\Hom_G(\mathcal{O}_{L,S,T}^\times, X_{L,S}) \xrightarrow{\nu_{S,I,*}}\Hom_G(\mathcal{O}_{L,S,T}^\times, Y_{L,I}) \xleftarrow{\nu_{I,*}} \Hom_G(\mathcal{O}_{L,S,T}^\times, \ZZ[G]^a)\]
where $\nu_{S,I}$ is the natural projection $X_{L,S} \to Y_{L,I}$ and $\nu_I$ the surjective homomorphism $\ZZ[G]^a \to Y_{L,I}$ sending each element $b_v$ of the standard basis $\{b_v\}_{v \in I}$ of $\ZZ[G]^a$ to $w_v$. We write $\Hom^I_G(\mathcal{O}_{L,S,T}^\times, X_{L,S})$ for the full pre-image of $\im(\nu_{I,*})$ under the map $\nu_{S,I,*}$.

\begin{remark}\label{easy split case}{\em If $S^a_{\rm min}$ comprises $a$ places that split completely in $L/K$, then $\wp^*_a(S) = \{I\}$ with $I := S^a_{\rm min}$ and, since the $G$-module $Y_{L,I}$ is free of rank $a$, one has 
\[ \Hom^I_G(\mathcal{O}_{L,S,T}^\times, X_{L,S}) = \Hom_G(\mathcal{O}_{L,S,T}^\times, X_{L,S}).\]
In general, the exponent of ${\rm cok}(\nu_{I,*})\cong {\rm Ext}^1_G(\mathcal{O}_{L,S,T}^\times,\ker(\nu_I))$  divides $|G|$ and so
\[ |G|\cdot \Hom_G(\mathcal{O}_{L,S,T}^\times, X_{L,S})\subseteq \Hom^I_G(\mathcal{O}_{L,S,T}^\times, X_{L,S})\]
for every 
$I$ in $\wp^\ast_a(S)$.}\end{remark}

\begin{lemma}\label{ell inv interpretation} For $a\in \mathbb{N}_0$ with $a< |S|$ and  $I\in \wp_a^\ast(S)$ one has
\begin{multline*}  \{({\wedge}_{v \in I}\varphi_v)(\varepsilon_{L/K,S}^I): \varphi_v\in\Hom_G(\mathcal{O}_{L,S,T}^{\times},\ZZ[G])\} =
\\ \xi(\ZZ[G])\cdot \{e_I(\theta^{(a)}_{L/K,S}(0))\cdot R(\psi): \psi \in \Hom^I_G(\mathcal{O}_{L,S,T}^{\times},X_{L,S})\}.\end{multline*}
\end{lemma}

\begin{proof} Fix $w_I \in(S\setminus I)_L$. For each $\varphi_\bullet := \{\varphi_v\}_{v \in I} \subset \Hom_G(\mathcal{O}_{L,S,T}^{\times},\ZZ[G])$ the definition (\ref{eta def}) of $\varepsilon_{L/K,S}^I$ combines with Lemma \ref{exam 2}(ii) and with the definition of $\pi_I$ to imply
\begin{align}\label{first calc} ({\wedge}_{v \in I}\varphi_v)(\varepsilon_{L/K,S}^I) = &\, ({\wedge}_{v \in I}\varphi_v)((\lambda^{\pi_I}_{L,S})^{-1}(e_I(\theta^{(a)}_{L/K,S}(0))\cdot {\wedge}_{v' \in I}w_{v'}))\\
= &\, e_I(\theta^{(a)}_{L/K,S}(0))\cdot ({\wedge}_{v \in I}\varphi_v)\left(({\bigwedge}^a_{\RR[G]} R_{L,S}^{-1})({\wedge}_{v' \in I}(w_{v'}-w_I))\right)\notag\\
= &\, e_I(\theta^{(a)}_{L/K,S}(0))\cdot {\rm Nrd}_{\RR[G]}(M(\varphi_\bullet)),\notag\end{align}
where we use the matrix 
\[ M(\varphi_\bullet) := ((\varphi_v\circ R_{L,S}^{-1})(w_{v'}-w_I))_{v,v' \in I} \in {\rm M}_{a}(\RR[G]),\]
so that the third equality in (\ref{first calc}) follows from \cite[Lem. 4.10]{bses}.

For $v \in I$ we now write $\psi_{v}$ for the composite homomorphism $\mathcal{O}_{L,S,T}^\times \to \ZZ[G]\to X_{L,S}$ where the first arrow is $\varphi_v$ and the second sends the trivial element of $G$ to $w_v-w_I$, and we then define $\psi_{\varphi_\bullet}$ to be the sum ${\sum}_{v \in I}\psi_v$. We note $\{e_I(w_v-w_I)\}_{v \in I}$ is an $\RR[G]e_I$-basis of $e_I(\RR\cdot X_{L,S})$ and that, with respect to this basis, the matrix of the endomorphism obtained by restricting $(\RR\otimes_\ZZ\psi_{\varphi_\bullet})\circ R_{L,S}^{-1}$ to $e_I(\RR\cdot X_{L,S})$ is equal to $e_I\cdot M(\varphi_\bullet)$. This implies that 
\[ e_I\cdot {\rm Nrd}_{\RR[G]}(M(\varphi_\bullet))= e_I\cdot {\rm Nrd}_{\RR[G]}((\RR\cdot\psi)\circ (R_{L,S})^{-1}) = e_I\cdot R(\psi_{\varphi_\bullet})\]
and hence, by (\ref{first calc}), that
\[ ({\wedge}_{v \in I}\varphi_v)(\varepsilon_{L/K,S}^I) =    e_I(\theta^{(a)}_{L/K,S}(0))\cdot R(\psi_{\varphi_\bullet}).\]

Next we note $e_I$ annihilates $\ker(\nu_{S,I,*})$ and hence that for any homomorphisms $\psi$ and $\psi'$ in $\Hom_G(\mathcal{O}_{L,S,T}^{\times},X_{L,S})$ one has $e_I\cdot R(\psi) = e_I\cdot R(\psi')$ whenever $\nu_{S,I,*}(\psi) = \nu_{S,I,*}(\psi')$.

To deduce the claimed equality from the last displayed equality it is thus enough to note $\Hom^I_G(\mathcal{O}_{L,S,T}^{\times},X_{L,S})$ is the subset of $\Hom_G(\mathcal{O}_{L,S,T}^{\times},X_{L,S})$ comprising all $\psi$ for which there exists a homomorphism $\psi_{\varphi_\bullet}$ of the above sort with $\nu_{S,I,*}(\psi) = \nu_{S,I,*}(\psi_{\varphi_\bullet})$.
\end{proof}

\section{Weil-Stark elements and the leading term conjecture}\label{5}

In this section we first review the leading term conjecture  ${\rm LTC}(L/K)$ for $L/K$, as stated in \cite{dals}, and the construction of Weil-Stark elements from \cite{bms1}. 

We then prove that, if suitable components of ${\rm LTC}(L/K)$ are valid, the corresponding generalised Stark elements are Weil-Stark elements.  

We continue to fix data $\Lambda$, $A$, $\A$, $\Pi$ and $\pi$ as in \S \ref{dataasin}.


\subsection{Review of the leading term conjecture}\label{eTNCreview} 

\subsubsection{}We write $D(\A)$ for the derived category of $\A$-modules and $D^{{\rm lf}}(\A)$ for its full triangulated subcategory comprising complexes isomorphic to a bounded complex of finitely generated locally-free $\A$-modules. (We note in passing that in the case $\A=\ZZ[G]$, a theorem of Swan \cite{swan} implies that $D^{{\rm lf}}(\ZZ[G])$ coincides with the full subcategory of $D(\ZZ[G])$ comprising complexes that are perfect.)

We fix sets $S$ and $T$ as in (\ref{data fix}) and use the `Weil-\'etale cohomology with compact support' complex $R\Gamma_{c}((\co_{L,S})_\W,\ZZ)$ constructed in \cite[\S2.2]{bks}. We recall that this complex has a natural interpretation in terms of Lichtenbaum's theory of Weil-\'etale cohomology (see \cite[Rem. 4.6]{bms1}) and defines an object of $D^{{\rm lf}}(\ZZ[G])$ up to canonical isomorphism. The complex 
\[ C_{L,S}:=R\Hom_\ZZ(R\Gamma_{c}((\co_{L,S})_\W,\ZZ),\ZZ)[-2]\]
is therefore also an object of $D^{{\rm lf}}(\ZZ[G])$ that is defined up to canonical isomorphism. 


Since each place $v$ in $T$ is unramified in $L$, there exists an exact sequence of $G$-modules 
\begin{equation}\label{resolve kappa} 0 \to \ZZ[G] \xrightarrow{x\to x(1-\sigma_v^{-1}{\rm N}v)} \ZZ[G] \to {\bigoplus}_{w \in \{v\}_L}\kappa^\times_w\to 0,\end{equation}
in which $\kappa_w$ denotes the residue field of $w$. This resolution implies $\bigl({\bigoplus}_{w \in \{v\}_L}\kappa^\times_w\bigr)[0]$ defines an object of $D^{\rm lf}(\ZZ[G])$ and hence that there exists an exact triangle in $D^{\rm lf}(\ZZ[G])$ of the form 
\begin{equation}\label{Ttriangle}C_{L,S,T}\longrightarrow C_{L,S}\longrightarrow\bigl({\bigoplus}_{w\in T_L}\kappa_w^\times\bigr)[0]\longrightarrow,\end{equation}
where the second arrow is the canonical morphism constructed in \cite[\S2.2]{bks}. In addition, by the 
argument of \cite[Lem. 4.5]{bms1}, one knows that 
\[ {_\Pi}C_{L,S,T}:=\Pi\otimes_{\ZZ[G]}^{\mathbb{L}}C_{L,S,T}\]
defines an object of $D^{\rm lf}(\A)$ that is acyclic outside degrees zero and one and such that $H^0({_\Pi}C_{L,S,T}) = {^\Pi}\mathcal{O}_{L,S,T}^\times$ and $H^1({_\Pi}C_{L,S,T})= {_\Pi}\mathcal{S}_S^T(L)^{{\rm tr}}$. 

\subsubsection{}
The equivariant $L$-function associated to the data $L/K, S$ and $T$ is defined by setting
$$\theta_{L/K,S,T}(s):={\sum}_{\chi\in \widehat G}L_{S,T}(\check \chi,s)e_\chi,$$
where $L_{S,T}(\check \chi,s)$ is the $S$-truncated $T$-modified Artin $L$-function for $\check\chi$.
The leading term of $\theta_{L/K,S,T}(s)$ at $s=0$ is then defined by setting
\begin{equation*}\label{thetastar}\theta_{L/K,S,T}^\ast(0):={\sum}_{\chi\in \widehat G}L_{S,T}^\ast(\check \chi,0)e_\chi\in\zeta(\RR[G])^\times.\end{equation*}

For any extension field $E$ of $\calF$ we write $A_E$ for the (semisimple) $E$-algebra $E\otimes_\calF A$ and $K_0(\A,A_E)$ for the relative algebraic $K$-group
 of the inclusion $\A\subset A_E$. This group is functorial in the pair $(\A,A_E)$ and also
 sits in a long exact sequence of relative $K$-theory 
\begin{equation}\label{relative}K_1(\A) \longrightarrow K_1(A_E)
\xrightarrow{\partial_{\A,A_E}} K_0(\A,A_E)\\
\longrightarrow K_0(\A).
\end{equation}
We recall that to each pair $(C,t)$ comprising an object $C$ of $D^{{\rm lf}}(\A)$ for which $E\otimes_\Lambda C$ is acyclic outside degrees $b$ and $b+1$ for some integer $b$ and an isomorphism of $A_E$-modules $t:E\otimes_\Lambda H^b(C) \cong E\otimes_\Lambda H^{b+1}(C)$ one can define a canonical `refined Euler characteristic' $\chi_{\A}(C,t)$ in $K_0(\A,A_E)$. We further recall that, if either $E=\CC$, or if $\Lambda\subset\QQ$ and $E\subseteq\RR$, then there exists a canonical `extended boundary homomorphism' of abelian groups 
\[ \delta_{\A,A_E}: \zeta(A_E)^\times \to K_0(\A,A_E)\]
with the property that the connecting homomorphism $\partial_{\A,A_E}$ in (\ref{relative}) factors as the composite 
$\partial_{\A,A_E} = \delta_{\A,A_E}\circ {\rm Nrd}_{A_E}$. 
If $\A = \ZZ[G]$ and $A_E = \RR[G]$, then we often abbreviate $\partial_{\A,A_\CC}$ and $\delta_{\A,A_\CC}$ to $\partial_G$ and $\delta_G$ respectively. 

\subsubsection{} We now provide a convenient reinterpretation of ${\rm LTC}(L/K)$.

\begin{proposition}\label{ltc2} The following claims are valid.\begin{itemize}
\item[(i)] ${\rm LTC}(L/K)$ is valid if and only if, in $K_0(\ZZ[G],\RR[G])$, one has
$$\delta_{G}(\theta_{L/K,S,T}^*(0))) = \chi_{\ZZ[G]}(C_{L,S,T},R_{L,S}).$$
\item[(ii)] If ${\rm LTC}(L/K)$ is valid then, in $K_0(\mathcal{A},A_\CC)$, one has
\begin{equation}\label{ltcPi} \delta_{\mathcal{A},A_\CC}(\iota_\Pi(\theta_{L/K,S,T}^*(0))) = \chi_{\mathcal{A}}({_\Pi}C_{L,S,T},R^\Pi_{L,S}),\end{equation}
where $R^\Pi_{L,S}$ denotes the the restriction of $\CC\cdot R_{L,S}$ to $\CC\cdot {_\Pi}\mathcal{O}_{L,S,T}^\times$.
\item[(iii)]  Assume that $A$ is a direct factor  of $\QQ[G]$ and (\ref{ltcPi}) is valid with $\Pi = \mathcal{A}$. Then, for every normal subgroup $H$ of $G$,  (\ref{ltcPi}) is also valid with $(L/K,A,\A,\A)$ replaced by $(L^H/K,\rho(A),\rho(\A),\rho(\A))$, where $\rho:\QQ[G]\to\QQ[G/H]$ is the canonical map.
\end{itemize}
\end{proposition}

\begin{proof} We set $C:=C_{L,S,T}$ and first explain how to deduce claim (ii) from claim (i).

We write $\mu^{\rm rel}_\Pi:K_0(\ZZ[G],\RR[G])\to K_0(\A,A_\CC)$ for the canonical homomorphism that is induced by the functor ${_\Pi}-$.
Then one need only use the commutative diagram
\begin{equation*}\begin{CD}
\zeta(\RR[G])^\times @> \delta_G  >> K_0(\ZZ[G],\RR[G])\\
@V \iota_\Pi VV @VV \mu^{\rm rel}_\Pi V\\
\zeta(A_\CC)^\times @> \delta_{\mathcal{A},A_\CC} >> K_0(\mathcal{A},A_\CC)
\end{CD}\end{equation*}
(following from the commutativity of (\ref{key commute}) and the naturality of connecting homomorphisms in relative $K$-theory) and the fact that the map $\mu^{\rm rel}_\Pi$ sends $\chi_{\ZZ[G]}(C,R_{L,S})$ to $\chi_{\mathcal{A}}({_\Pi}C,R^\Pi_{L,S})$

We now proceed to prove claim (i). 
We fix, as we may, a finite set $S'$ of places of $K$ that split completely in $L/K$, disjoint from $S\cup T$, and large enough that ${\rm Cl}_\Sigma(L)$ vanishes for $\Sigma:=S\cup S'$ (see \cite[Lem. 5.1.1]{dals}).
Then, by making explicit the result of \cite[Prop. 2.4 (ii)]{bks}, one finds that there is a canonical exact triangle in $D^{\rm lf}(\ZZ[G])$ 
\begin{equation}\label{Sigmatriangle}C_{L,S}\longrightarrow C_{L,\Sigma}\longrightarrow Y_{L,S'}[0]\oplus Y_{L,S'}[-1]\longrightarrow,\end{equation}
that induces the exact cohomology sequence
$$0\to\co_{L,S}^\times\to\co_{L,\Sigma}^\times\to Y_{L,S'}\to\mathcal{S}_S(L)^{\rm tr}\to X_{L,\Sigma}\to Y_{L,S'}\to 0,$$
where the first arrow is inclusion, the second arrow is induced by taking valuations at each place in $S'_L$, the third arrow is the composition $Y_{L,S'}\to {\rm Cl}_S(L)\to \mathcal{S}_S(L)^{\rm tr}$, the fourth arrow is the composition $\mathcal{S}_S(L)^{\rm tr}\to X_{L,S}\to X_{L,\Sigma}$ and the fifth arrow is canonical.


We next recall that ${\rm LTC}(L/K)$ is formulated in \cite[\S 6.1]{dals} as the equality
$$\delta_{G}(\theta_{L/K,\Sigma}^*(0))) = \chi_{\ZZ[G]}(C_{L,\Sigma},R_{L,\Sigma})$$
and that, setting $\pi_{S'} := {\prod}_{v\in S'}{\rm Nrd}_{\RR[G]}(\log({\rm N}v))$,  one has
\[ \delta_{G}(\theta_{L/K,\Sigma}^*(0))) =  \delta_G\bigl(\pi_{S'}\bigr)+\delta_{G}(\theta_{L/K,S}^*(0)))=\delta_G\bigl(\pi_{S'}\bigr)-\delta_G(\gamma_T)+\delta_{G}(\theta_{L/K,S,T}^*(0))).\]

It will therefore be enough to show that
\begin{equation}\label{enoughECs}\chi_{\ZZ[G]}(C_{L,\Sigma},R_{L,\Sigma})= \delta_G\bigl(\pi_{S'}\bigr)-\delta_G(\gamma_T)+\chi_{\ZZ[G]}(C_{L,S,T},R_{L,S}).\end{equation}
In turn, this equality will follow upon applying the additivity properties of refined Euler characteristics (as in \cite[Lem. A.1.2]{dals}) to the exact triangles (\ref{Sigmatriangle}) and (\ref{Ttriangle}), after recalling the exactness of the sequences of the form (\ref{resolve kappa}) and observing the commutativity of the canonical exact diagram of $\RR[G]$-modules
\begin{equation*}
\begin{CD}
0 @> >>  \RR\cdot\co_{L,S}^\times @>  >>  \RR\cdot\co_{L,\Sigma}^\times @>  >> \RR\cdot Y_{L,S'} @> >> 0\\
@. @V R_{L,S} VV @V R_{L,\Sigma} VV @VV {\rm Log}_{S'} V @.\\
0 @> >>  \RR\cdot X_{L,S} @>  >> \RR\cdot X_{L,\Sigma}@>  >> \RR\cdot Y_{L,S'} @> >> 0,
\end{CD}\end{equation*}
where ${\rm Log}_{S'}$ maps a place $w$ above a place $v\in S'$ to $\log({\rm N}v)w$.

Indeed, if we set $\kappa_T^\times:={\bigoplus}_{w\in T_L}\kappa_w^\times$, then one has
\begin{align*}\chi_{\ZZ[G]}(C_{L,\Sigma},R_{L,\Sigma})=&\chi_{\ZZ[G]}(Y_{L,S'}[0]\oplus Y_{L,S'}[-1],{\rm Log}_{S'})+\chi_{\ZZ[G]}(C_{L,S},R_{L,S})\\
=&\delta_G({\rm Nrd}_{\RR[G]}({\rm Log}_{S'}))+\chi_{\ZZ[G]}(C_{L,S},R_{L,S})\\
=&\delta_G\bigl(\pi_{S'}\bigr)+\chi_{\ZZ[G]}(\kappa_T^\times[0],0)+\chi_{\ZZ[G]}(C_{L,S,T},R_{L,S})\\
=&\delta_G\bigl(\pi_{S'}\bigr)-\chi_{\ZZ[G]}(\kappa_T^\times[-1],0)+\chi_{\ZZ[G]}(C_{L,S,T},R_{L,S})\\
=&\delta_G\bigl(\pi_{S'}\bigr)-\delta_G(\gamma_T)+\chi_{\ZZ[G]}(C_{L,S,T},R_{L,S}).
\end{align*}
This verifies the equality (\ref{enoughECs}) and thus completes the proof of claim (ii).

Finally, we observe that in the setting of claim (iii), the map $\iota_{\rho(\A)}$ factors through $\iota_\Pi$. Given this observation, and the result of claim (ii), the result of claim (iii) can be derived as a direct consequence of the (well-known) behaviour of $\theta_{L/K,S,T}^*(0)$, $C_{L,S,T}$ and $R_{L,S}$ under change of extension. 
\end{proof}

\begin{remark}{\em If $L$ is a number field, then ${\rm LTC}(L/K)$ is equivalent to the equivariant Tamagawa number conjecture for $(h^0({\rm Spec}(L)),\ZZ[G])$ (cf. \cite[Rem. 6.1.1(i)]{dals}). In this way, the result of Proposition \ref{ltc2}(iii) corresponds to a property of equivariant Tamagawa numbers under Galois descent (cf. \cite[Rem. 6.1.1(ii)]{dals}).}\end{remark} 

\begin{remark}\label{Picomponent}{\em We say that `${\rm LTC}(L/K,\Pi)$ is valid' if the displayed equality in Proposition \ref{ltc2}(ii) is valid (with the order $\mathcal{A}$, and hence algebra $A$, implied by the fact $\Pi$ is regarded as a left $\mathcal{A}$-module). In particular, if $\Pi = \ZZ[G]$ (as a $(\ZZ[G],\ZZ[G])$-bimodule), then ${\rm LTC}(L/K,\Pi)$ is valid if and only if ${\rm LTC}(L/K)$ is valid. However, if, in the setting of Examples \ref{exam1}(iv) and \ref{exam 2r}(ii) and Remark \ref{stark context}, one has $\Pi = \Pi_\psi$ for $\psi\in \widehat{G}$, then the argument of \cite[\S12.1]{dals} shows that ${\rm LTC}(L/K,\Pi)$ is valid if and only if $\psi$ validates the `Strong-Stark Conjecture' of Chinburg \cite{chinburgomega}. 
It follows that ${\rm LTC}(L/K,\Pi)$ depends critically on $\Pi$ 
and can be much weaker than ${\rm LTC}(L/K)$. }\end{remark}

\subsection{Weil-Stark elements}
\subsubsection{}
We first recall the construction of Weil-Stark elements.  
%
%
We write $D^{\rm lf}(\A)_{\rm is}$ for the subcategory of $D^{\rm lf}(\A)$ in which morphisms are restricted to be isomorphisms, and $\calP(C)$ for the category of graded invertible modules over a commutative ring $C$ (though we omit any reference to the gradings). Then \cite[Th. 5.2]{bses} gives a `reduced determinant functor' $$\D_{\A}:D^{\rm lf}(\A)_{\rm is}\to\calP(\xi(\A)),$$ that respects exact triangles in $D^{\rm lf}(\A)$ and is such that 
$\D_{\A}( M[0]) =  {{{{\bigcap}}}}_{\A}^aM$
for every locally-free $\A$-module $M$ of rank $a$. 
The approach of loc. cit. also defines associated functors
$$\D^\diamond_A:{\rm Mod}(A)\to\calP(\zeta(A))\,\,\,\,\,\,\text{ and }\,\,\,\,\,\,\D_A:D^{\rm lf}(A)_{\rm is}\to\calP(\zeta(A)),$$
where ${\rm Mod}(A)$ is the category of finitely generated $A$-modules. The determinant functors $\D_{\A}$ and $\D_A$ are then compatible with extension of scalars. 

We now write ${_\Pi}C$ for the complex $\Pi\otimes_{\ZZ[G]}^{\mathbb{L}}C_{L,S}$ in $D(\A)$ and set $r := {\rm rk}_\A(Y_\pi)$. For each prime ideal $\frp$ of $\Lambda$ we then fix, as we may, an $\A_{(\frp)}$-basis $\underline{b}_{\pi,\frp}$ 
of $Y_{\pi,(\frp)}$. This choice of basis induces a composite homomorphism of $\xi(\A_{(\frp)})$-modules
\begin{align*}\label{Thetap}\Theta_{\underline{b}_{\pi,\frp}}:\D_{A}(\calF\cdot {_\Pi}C)
\cong&\,\,\D^\diamond_{A}\bigl(\calF\cdot{^\Pi}\mathcal{O}_{L,S}^\times\bigr)\otimes_{\zeta(A)}\D^\diamond_{A}\bigl(\calF\cdot{_\Pi}\mathcal{S}_S(L)^{\rm tr}\bigr)^{-1}\\ \notag
\to&\,\,\D^\diamond_{Ae_\pi}\bigl(e_\pi(\calF\cdot{^\Pi}\mathcal{O}_{L,S}^\times)\bigr)\otimes_{\zeta(A)e_\pi}\D^\diamond_{Ae_\pi}\bigl(e_\pi(\calF\cdot{_\Pi}X_{L,S})\bigr)^{-1}\\ \notag
\cong&\,\,\D^\diamond_{Ae_\pi}\bigl(e_\pi(\calF\cdot{^\Pi}\mathcal{O}_{L,S}^\times)\bigr)\otimes_{\zeta(A)e_\pi}\D^\diamond_{Ae_\pi}\bigl(e_\pi(\calF\cdot Y_\pi)\bigr)^{-1}\\ \notag
\cong&\,\,\D^\diamond_{Ae_\pi}\bigl(e_\pi(\calF\cdot{^\Pi}\mathcal{O}_{L,S}^\times)\bigr)\\ \notag
=&\,\,e_\pi\bigl(\calF\cdot{{{\bigcap}}}^r_{\A}{^\Pi}\mathcal{O}_{L,S}^\times\bigr).
\end{align*}
Here the first isomorphism is induced by the `passage to cohomology' map from \cite[Prop. 5.14 (i)]{bses} for the complex of $\A$-modules ${_\Pi}C$, 
the arrow is induced by multiplication by $e_\pi$ and by the exact sequence (\ref{selmer lemma II}), the second isomorphism is induced by $\pi$, and the third isomorphism is induced by the canonical isomorphism
\begin{equation*}\label{diamond}\zeta(A)e_\pi\cong\D^\diamond_{Ae_\pi}\bigl(e_\pi(\calF\cdot Y_{\pi})\bigr),\quad e_\pi\mapsto e_\pi\wedge_{i=1}^{i=r}b_{\frp,i}.\end{equation*}


The lattice $\W_\frp:=\Theta_{\underline{b}_{\pi,\frp}}\bigl(\D_{\A_{(\frp)}}({_\Pi}C_{(\frp)})\bigr)$ is independent of the choice of basis $\underline{b}_{\pi,\frp}$ of $Y_{\pi,(\frp)}$. It is also proved in \cite[Lem. 4.7]{bms1} that the module of `Weil-Stark elements' 
$$\W_{L/K,S}^{\Pi,\pi}:={{{\bigcap}}}_{\frp}\W_{\frp}$$ 
is a full $\xi(\A)e_\pi$-submodule of $e_\pi(\calF\cdot{{{\bigcap}}}^r_{\A}{^\Pi}\mathcal{O}_{L,S}^\times)$ that, for each $\frp$, has  $(\W_{L/K,S}^{\Pi,\pi})_{(\frp)}=\W_\frp$. 

%
%
%

\begin{remark}\label{ntWSe}{\em  The above construction implies $\W^{\Pi,\pi}_{L/K,S}\not= (0)$ if and only if 
${\rm rk}_\A(Y_\pi)$ is equal to the maximal possible rank $r_{L,S}^\Pi$ of a locally-free $\mathcal{A}$-quotient of ${_{\Pi}}X_{L,S}$ (cf. Remark \ref{cases of r}).}\end{remark}

\subsubsection{}
The next result establishes a link between these Weil-Stark elements and the generalised Stark elements introduced in  \S\ref{hnase section}.

\begin{theorem}\label{eTNC} If ${\rm LTC}(L/K;\Pi)$ is valid, then for each map $\pi$ as in (\ref{surj Pi map}) and each $r^\Pi_{L,S}$-tuple $x_\bullet$ in $Y_\pi$, the generalised Stark element $\varepsilon_{x_\bullet}^{\pi}$ belongs to $\W_{L/K,S}^{\Pi,\pi}$.\end{theorem}

\begin{proof} The regulator $R^\Pi_{L,S}$ induces a canonical isomorphism of (graded) $\zeta(A_\CC)$-modules $\lambda^\Pi_{L,S}: \D_{A_\CC}(\CC\cdot ({_\Pi}C_{L,S,T})) \to \zeta(A_\CC)$ and we set 
$$\calL^{\Pi}_{L,S,T}:=\iota_\Pi(\theta_{L/K,S,T}^*(0))^{-1}\cdot \lambda^\Pi_{L,S}(\D_\A({_\Pi}C_{L,S,T}))\subset \zeta(A_\CC).$$ 

We next note that Proposition \ref{ltc2}(ii) combines with the general result of \cite[Th. 4.8(i)]{bses2} to imply that if ${\rm LTC}(L/K;\Pi)$ is valid, then the pre-image under $\lambda^\Pi_{L,S}$ of $\iota_\Pi(\theta^\ast_{L/K,S,T}(0))$ is a $\xi(\A)$-basis of the determinant module $\D_\A({_\Pi}C_{L,S,T})$ and hence that 
\begin{equation}\label{ltc3} \calL^{\Pi}_{L,S,T} =\xi(\A).\end{equation}
Given this, the result of Theorem \ref{eTNC} is completed by the result of Lemma \ref{StarkWeil} below.\end{proof}

\begin{lemma}\label{StarkWeil} For each $\pi$ and each $r^\Pi_{L,S}$-tuple $x_\bullet$ in $Y_\pi$ one has $\calL^{\Pi}_{L,S,T}\cdot\varepsilon_{x_\bullet}^{\pi}\in \W^{\Pi,\pi}_{L/K,S}$.
\end{lemma}
\begin{proof} By Lemma \ref{stick vanishing}, we may and will assume that ${\rm rk}_\A(Y_\pi)=r:=r^\Pi_{L,S}$.
It is enough to fix an arbitrary prime $\frp$ of $\Lambda$ and prove  $\calL^{\Pi}_{L,S,T}\cdot\varepsilon_{x_\bullet}^{\pi}\in \W_\frp$.

We fix, as we may, an $\A_{(\frp)}$-basis $\underline{b}_{\pi,\frp}=\{b_{\frp,i}\}_{1\leq i\leq r}$ of $Y_{\pi,(\frp)}$. By \cite[Lem. 4.13]{bses}, there is $\mu$ in $\xi(\A_{(\frp)})$ for which $\wedge_{i=1}^{i=r}x_i=\mu\cdot\wedge_{i=1}^{i=r}b_{\frp,i}$. But then
$$\varepsilon_{x_\bullet}^{\pi}=\iota_\Pi(\theta^\pi_S(0))\cdot(\lambda^{\pi}_{L,S})^{-1}(\mu\cdot\wedge_{i=1}^{i=r}b_{\frp,i})=\mu\cdot\iota_\Pi(\theta^\pi_S(0))\cdot(\lambda^{\pi}_{L,S})^{-1}(\wedge_{i=1}^{i=r}b_{\frp,i}).$$
We set $\varepsilon_{\underline{b}_{\pi,\frp}}:=\iota_\Pi(\theta^\pi_S(0))\cdot(\lambda^{\pi}_{L,S})^{-1}(\wedge_{i=1}^{i=r}b_{\frp,i})$ and deduce that it is enough to prove that 
\begin{equation}\label{equalityforbasis}(\calL^{\Pi}_{L,S,T})_{(\frp)}\cdot\varepsilon_{\underline{b}_{\pi,\frp}}=\W_\frp.\end{equation}

Writing $\sigma_{\underline{b}_{\pi,\frp}}:e_\pi\bigl({\bigwedge}_{A_\CC}^{r}\CC\cdot Y_\pi\bigr)\stackrel{\sim}{\longrightarrow}e_\pi\zeta(A_\CC)$ for the isomorphism that maps $e_\pi\wedge_{i=1}^{i=r}b_{\frp,i}$ to $e_\pi$, there is commutative square
\begin{equation}\label{detsquare}
\xymatrix{ {\rm d}_{A_\CC}(\CC\cdot {_\Pi}C_{L,S,T})  \ar@{->}[rr]^{e_\pi\lambda_{L,S}^\Pi} \ar@{->}[d]_{(\Theta_{\underline{b}_{\pi,\frp}})_\CC}  & & e_\pi\zeta(A_\CC)  \\
e_\pi({{\bigwedge}}_{A_\CC}^r \CC\cdot {^\Pi}\mathcal{O}_{L,S}^\times)   \ar@{->}[rr]^{\lambda^{\pi}_{L,S}} &  &  e_\pi\bigl(\bigwedge_{A_\CC}^{r}\CC\cdot Y_\pi\bigr)  \ar@{->}[u]_{\sigma_{\underline{b}_{\pi,\frp}}} .
} 
\end{equation}

The claimed equality (\ref{equalityforbasis}) is now valid as a consequence of the injectivity of both $\sigma_{\underline{b}_{\pi,\frp}}$ and $\lambda^{\pi}_{L,S}$ and of the equality of lattices 

\begin{align*}\sigma_{\underline{b}_{\pi,\frp}}\bigl(\lambda^{\pi}_{L,S}\bigl((\calL^{\Pi}_{L,S,T})_{(\frp)}\cdot\varepsilon_{\underline{b}_{\pi,\frp}}\bigr)\bigr)
=\,&(\calL^{\Pi}_{L,S,T})_{(\frp)}\cdot\sigma_{\underline{b}_{\pi,\frp}}(\lambda^{\pi}_{L,S}(\varepsilon_{\underline{b}_{\pi,\frp}}))\\
=\,&(\calL^{\Pi}_{L,S,T})_{(\frp)}\cdot(e_\pi\cdot\iota_\Pi(\theta^{\pi}_S(0)))\\
=\,&\lambda^\Pi_{L,S}(\D_{\A_{(\frp)}}(({_\Pi}C_{L,S,T})_{(\frp)}))\cdot e_\pi\cdot\iota_\Pi(\theta^*_{L/K,S,T}(0)^{-1}\cdot\theta^{\pi}_S(0))\\
=\,&e_\pi\lambda^\Pi_{L,S}(\D_{\A_{(\frp)}}(({_\Pi}C_{L,S,T})_{(\frp)}))\cdot\iota_\Pi(\gamma_{T})^{-1}\\
=\,&e_\pi\lambda^\Pi_{L,S}(\D_{\A_{(\frp)}}(({_\Pi}C_{L,S})_{(\frp)}))\\
=\,&\sigma_{\underline{b}_{\pi,\frp}}\left(\lambda^{\pi}_{L,S}\left(\W_\frp\right)\right).
\end{align*}
Here the fourth equality uses Lemma \ref{stick vanishing} and the fact $\theta^*_{L/K,S,T}(0) = \gamma_T\cdot \theta^{*}_{L/K,S}(0)$ while 
the sixth equality from the commutativity of (\ref{detsquare}). To justify the fifth equality and thus complete the proof, one must simply apply the functor ${\rm d}_\A(-)$ to the exact triangle
\begin{equation*}\label{Ttriangle2} {_\Pi}C_{L,S,T} \longrightarrow {_\Pi}\,C_{L,S} \longrightarrow {_\Pi}\bigl({\bigoplus}_{w\in T_L}\kappa_w^\times\bigr)[0]\longrightarrow,\end{equation*}
in $D^{\rm lf}(\A)$ that is induced by (\ref{Ttriangle}) and 
note that
\begin{equation*}\label{dT} \D_{\A}({_\Pi}C_{L,S,T}) = \D_\A\bigl({_\Pi}\bigl({\bigoplus}_{w\in T_L}\kappa_w^\times\bigr)[-1]\bigr)\cdot \D_\A( {_\Pi}C_{L,S}) = \iota_\Pi(\gamma_{T})\cdot \D_\A({_\Pi}C_{L,S}),\end{equation*}
where the second equality follows from the resolutions (\ref{resolve kappa}) and the commutativity of (\ref{key commute}). 
\end{proof}

\begin{remark}\label{dals improve4}{\em If $\Pi = \Pi_\psi$ for some character $\psi$ in $\widehat{G}$, then ${\rm LTC}(L/K,\Pi_\psi)$ is valid if and only if $\psi$ validates the Strong-Stark Conjecture (cf. Remark \ref{Picomponent}) and the corresponding generalised Stark elements encompass the elements that occur in \cite[Conj. 2.6.1]{dals} and were studied earlier in \cite{chin, stark0, stark} (cf. Remark \ref{stark context}). In this setting, therefore, Theorem \ref{eTNC} incorporates a refinement of \cite[Conj. 2.6.1]{dals} and hence of the original conjectures of Stark and Chinburg. This case will be considered in detail in \S\ref{cs}.
}
\end{remark}

\subsubsection{}
To end this section, we derive a useful consequence of Theorem \ref{eTNC}. To do this we fix $a \in \mathbb{N}_0$ with  $a< |S|$ and $I\in \wp_a^\ast(S)$ and use the notation of Example \ref{exam1}(ii), Lemma \ref{exam 2}, Remark \ref{rubin-stark context} and Lemma \ref{ell inv interpretation}. We then set 
\[\W^{I}_{L/K,S}:=\W^{\Pi,\pi_I}_{L/K,S} \subset \QQ\cdot{{{\bigcap}}}^a_{\Lambda[G]e_I}{^{\Pi}}\mathcal{O}_{L,S}^\times\subseteq e_I\bigl(\bigwedge_{\QQ[G]}^{a}\QQ\cdot\co^\times_{L,S}\bigr).\]

\begin{corollary}\label{Weil-Linvariants} Fix $a\in \mathbb{N}_0$ with $a< |S|$ and $I\in \wp_a^\ast(S)$. Then, if ${\rm LTC}(L/K;\Lambda[G]e_I)$ is valid, one has $\W^{I}_{L/K,S}=\xi(\Lambda[G]e_I)\cdot\varepsilon^{I}_{L/K,S}$. In particular, in this case one has
\begin{multline}\label{first equiv}  \{({\wedge}_{i=1}^{i=a}\varphi_i)(\omega):\,\omega\in\W^{I}_{L/K,S},\,  \varphi_i\in\Hom_G(\mathcal{O}_{L,S,T}^{\times},\ZZ[G])\} =
\\ \xi(\Lambda[G])\cdot\{e_I(\theta^{(a)}_{L/K,S}(0))\cdot R(\psi): \psi \in \Hom^I_G(\mathcal{O}_{L,S,T}^{\times},X_{L,S})\}.\end{multline} 
\end{corollary}

\begin{proof} At the outset we claim that if for any $a\in \mathbb{N}_0$ with $a<|S|$ and $I\in \wp_a^*(S)$, one writes $\Pi$ for the $(\Lambda[G]e_I,\Lambda[G])$-bimodule $\Lambda[G]e_I$, then 

\[ \calL^{\Pi}_{L,S,T}\cdot\varepsilon^{I}_{L/K,S}=\W^{I}_{L/K,S}.\]
To see this, we observe $\underline{x}_I=\{w_v:\,v\in I\}$ is a $\Lambda[G]e_I$-basis of ${_\Pi}Y_{L,I}$ and so induces a  $\Lambda_{(\frp)}[G]e_I$-basis of $({_\Pi}Y_{L,I})_{(\frp)}$ for every prime ideal $\frp$ of $\Lambda$. The family of equalities (\ref{equalityforbasis}) for each such $\frp$ then combine to imply the above equality.

To prove the first claim, it is now enough to note that the validity of ${\rm LTC}(L/K;\Lambda[G]e_I)$ implies the equality (\ref{ltc3}), which translates to the current setting to combine with the above displayed equality to imply  $\W^{I}_{L/K,S}=\calL^{\Pi}_{L,S,T}\cdot\varepsilon^{I}_{L/K,S}=\xi(\Lambda[G]e_I)\cdot\varepsilon^{I}_{L/K,S}.$ Then (\ref{first equiv}) follows immediately upon combining the last equality with Lemma \ref{ell inv interpretation}. \end{proof}

\section{Arithmetic properties of Artin $\mathscr{L}$-invariants}\label{integrality section}

In this section we combine Theorem \ref{eTNC} with results from \cite{bms1} to derive a range of concrete predictions concerning the arithmetic properties of Artin $\mathscr{L}$-invariants. 

\subsection{Statement of the main result}\label{6.1}

Throughout this section, we let $\Lambda$ denote either $\ZZ$ or a localisation thereof at some prime number and we abbreviate $\Lambda\otimes_{\ZZ}-$ to $\Lambda\cdot -$.

We fix sets $S$ and $T$ as in (\ref{data fix}) and $a\in\NN_0$ with $a< |S|$ and set $\theta_{L/K,S,T}^{(a)}(0):=\gamma_{T}\cdot\theta_{L/K,S}^{(a)}(0)$. For $I\in \wp_a^\ast(S)$ we use the notation of Lemma \ref{ell inv interpretation} and set 
$$\Theta^{I}_{L/K,S,T}:=\xi(\Lambda[G])\cdot\{e_I(\theta^{(a)}_{L/K,S,T}(0))\cdot R(\psi): \psi \in \Hom^I_G(\mathcal{O}_{L,S,T}^{\times},X_{L,S})\}.$$


Finally, for $v\in S$ we define an ideal of $\xi(\Lambda[G])$ by setting 
\[ \mathfrak{m}^a_v(L/K) := \{ x\in \xi(\Lambda[G]): x\cdot (1-e_v+e_{\bf 1})e_{(a),S}\cdot\Lambda\cdot{\rm Fit}^{{\rm tr},a}_{\ZZ[G]}(\mathcal{S}_S^T(L)) \subseteq \Lambda\cdot{\rm Fit}^{{\rm tr},a}_{\ZZ[G]}(\mathcal{S}_S^T(L))\}.\] 
%

\begin{theorem}\label{main conj} Fix $a\in\NN_0$ with $a< |S|$ and assume 
 ${\rm LTC}(L/K;\Lambda[G]e_{(a),S})$ is valid.

\begin{itemize}
\item[(i)] Then, for each $v\in S$, there exists a locally-quadratic presentation $\tilde h=\tilde h_v$ of $\mathcal{S}_S^T(L)^{\rm tr}$ that is independent of $a$ and such that 
\begin{equation}\label{prediction1} (1-e_v+e_{\bf 1})e_{(a),S}\left(\Lambda\cdot{\rm Fit}_{\ZZ[G]}^a(\tilde h)\right) = c^a_{S,v}\Theta_{K}^a\oplus {\bigoplus}_{I\in \wp_a(S,v)} \Theta_{L/K,S,T}^I,\end{equation}
where, on the right hand side, we use the integer $c^a_{S,v}$ of Theorem \ref{main result alg} and the $\Lambda$-submodule $\Theta_{K}^a$ of $\RR e_{\bf 1}$ specified in Remark \ref{ThetaKK} below.

\item[(ii)] For any $v\in S$ one has $$\mathfrak{m}_{v}^a(L/K)\cdot\iota_\#\bigl({\sum}_{I\in \wp_a(S,v)} \Theta_{L/K,S,T}^I\bigr)\subseteq \Lambda\cdot{\rm Fit}^{{\rm tr},a}_{\ZZ[G]}(\mathcal{S}_S^T(L)).$$ 

\item[(iii)] If $a<|S|-1$ then for any $v$ in $S$, $x$ in $\mathfrak{m}_{v}^a(L/K)$ and $\varphi$ in the intersection $\bigcap_{I\in \wp_a(S,v)}\Hom_G^{I}(\co_{L,S,T}^\times,X_{L,S})$, one has
\begin{equation}\label{Linvproduct}(1-e_v)e_{(a),S}\cdot x\cdot \iota_\#(\mathscr{L}^a(\varphi))\,\in \,\iota_\#(\gamma_T^{-1})\cdot\left(\Lambda\cdot{\rm Fit}^{{\rm tr},a}_{\ZZ[G]}(\mathcal{S}_S^T(L))\right).\end{equation}
In particular, this containment is valid for any $\varphi$ in $|G|\cdot\Hom_G(\co_{L,S,T}^\times,X_{L,S})$.

\item[(iv)] Fix $I\in \wp^\ast_a (S)$ and $v_0\in S\setminus I$ and set $S' :=S_K^\infty\cup I\cup\{v_0\}$. Then $$|N_I|\cdot\delta(\Lambda[G])\cdot\Theta^{I}_{L/K,S,T}\subseteq \Lambda\cdot{\rm Ann}_{\ZZ[G]}({\rm Cl}^T_{S'}(L)),$$
where $N_I$ is the normal closure in $G$ of $\bigcup_{v\in I}G_v$ and $\delta(\Lambda[G])$ as in \S\ref{Whitehead section}.
\end{itemize}
\end{theorem}

\begin{remark}\label{ThetaKK}{\em In (i) we set $\Theta^{a}_{K}:=0$ if $a<|S|-1$ and $\Theta^{|S|-1}_{K} := \theta^{(a)}_{K/K,S,T}(0)\cdot R_{K,S,T}^{-1}\cdot\Lambda e_{\bf 1}$, with $R_{K,S,T}\in\RR^\times$ the determinant of the Dirichlet regulator $\RR\cdot \mathcal{O}_{K,S}^\times \cong \RR\cdot X_{K,S}$ with respect to any $\ZZ$-bases of $\mathcal{O}_{K,S,T}^\times$ and $X_{K,S}$.
}\end{remark}

\begin{remark}\label{dals improve cor} {\em Assume $L/K$ is non-trivial and fix a proper subset $V$ of $S$ comprising places that split completely in $L$. Set $a := |V|$ and fix $v \in S\setminus V$. Then $\wp_a(S,v) = \{V\}$, $S_{\rm min}^a = V$ and $e_{(a),S} = 1$ and so Lemma \ref{fit vanishing}(iii) implies the left hand side of (\ref{prediction1}) is equal to  ${\rm Fit}_{\ZZ[G]}^a(\tilde h)$. In this case, $c^a_{S,v}\cdot \Theta_{K}^a = 0$ and so Lemma \ref{ell inv interpretation} implies that (\ref{prediction1}) is valid if and only if  

\begin{equation}\label{Fitrhv}\notag\Lambda\cdot{\rm Fit}^a_{\ZZ[G]}(\tilde h)= \Lambda\cdot\{({\wedge}_{j=1}^{j=a}\varphi_j)(\gamma_T\cdot\varepsilon_{L/K,S}^V): \varphi_j\in\Hom_G(\mathcal{O}_{L,S,T}^{\times},\ZZ[G])\}.\end{equation}
This equality (and Remark \ref{rubin-stark context}) implies Theorem \ref{main conj}(i) recovers a result \cite[Th. 1.5]{bks} proved for abelian extensions by Kurihara, Sano and the first author. In general, this equality also implies that $\gamma_T\cdot\varepsilon_{L/K,S}^V\in \Lambda\cdot\bigcap_{\ZZ[G]}^{a}\co_{L,S}^\times$, as conjectured by Sano and the first author in \cite{bses2} (as a generalisation of the conjectures formulated by Rubin \cite{R} and Popescu \cite{Pop}). 
}
\end{remark}

\begin{remark}\label{dals improve}{\em Theorem \ref{main conj} refines the predictions of \cite[Conj. 2.4.1]{dals} (and hence, by specializing to the case $r=0$, of the `non-abelian Brumer-Stark Conjecture' formulated by Nickel in \cite{nickel-aif}), and its proof also improves upon the approach of \cite{dals} in two ways. Firstly, it avoids an important technical hypothesis (concerning the cohomological-triviality of roots of unity) that is required by the arguments used in \cite{dals}; secondly, it gives a much more direct approach to the main result (Th. 2.4.1) of \cite{dals}, thereby avoiding reliance on the extensive computations in \cite[\S6.4]{dals} and the technical constructions of Ritter and Weiss in \cite{RW0}. }
\end{remark}

\begin{remark}\label{dualversionuseful}
{\em Assume ${\rm LTC}(L/K;\Lambda[G]e_{(a),S})$ and, for $v\in S$, define an $\xi(\Lambda[G])$-ideal
\[ \mathfrak{n}^a_v(L/K) := \{x\in \xi(\Lambda[G]): x\cdot (1-e_v+e_{\bf 1})e_{(a),S}\cdot\Lambda\cdot{\rm Fit}^{{\rm tr},a}_{\ZZ[G]}(\mathcal{S}_S^T(L)) \subseteq \xi(\Lambda[G])\}.\] 
Then our argument proves 
%
%
$\mathfrak{n}_{v}^a(L/K)\cdot\iota_\#\bigl({\sum}_{I\in \wp_a(S,v)} \Theta_{L/K,S,T}^I\bigr)\subseteq\xi(\Lambda[G])$ for all $I\in \wp_a(S,v)$. In addition, in the setting of claim (iii), if $x\in \mathfrak{n}^a_v(L/K)$ then the left-hand side of (\ref{Linvproduct}) belongs to $\iota_\#(\gamma_T^{-1})\cdot\xi(\Lambda[G])$.
}\end{remark}

\begin{remark}\label{dals improve3}{\em 
If $G$ is abelian, the claims in Theorem \ref{main conj}, together with those outlined in Remark \ref{dualversionuseful}, 
recover the central conjecture (Conjecture 4.3) of Livingstone Boomla and the first author in \cite{dbalb}. We recall that the latter conjecture refines and extends the conjectures of Emmons and Popescu in \cite{EP} and Valli\`{e}res in \cite{vallieres} and can be investigated numerically using the methods developed by Bley in \cite{bley st} and by McGown, Sands and Valli\`{e}res in \cite{msv}. }\end{remark}




\begin{remark} \label{exp cor}
{\em For $a\in \mathbb{N}_0$ set $\mathcal{F}_a := \Lambda\cdot{\rm Fit}_{\ZZ[G]}^{{\rm tr},a}(\mathcal{S}_S^T(L))$ and, for each $v\in S$, also $\mathcal{F}_{a,v}' := \Lambda\cdot{\rm Fit}_{\ZZ[G]}^{{\rm tr},a}(h_v)$ with $h_v$ a presentation of $\mathcal{S}_S^T(L)$ as in Theorem \ref{main result alg}. 
Then, since $\mathcal{F}_{a,v}' \subseteq \mathcal{F}_a$, Theorem \ref{main result alg} combines with the definitions of the ideals $\mathfrak{n}^a_v(L/K)$ and $\mathfrak{m}^a_v(L/K)$ to imply that for $I$ in $\wp_a(S,v)$ one has $\mathfrak{n}^a_v(L/K)\cdot e_I\cdot\mathcal{F}'_{a,v} \subseteq \xi(\Lambda[G])$ and $\mathfrak{m}^a_v(L/K)\cdot e_I\cdot\mathcal{F}'_{a,v} \subseteq \mathcal{F}_a.$ 

An explicit computation of $\mathfrak{n}_{v}^a(L/K)$ and $\mathfrak{m}_{v}^a(L/K)$ in a general setting would be rather involved since it requires some knowledge of the Fitting invariants of $\mathcal{S}_S^T(L)$. It is, nevertheless, straightforward to construct non-zero elements of 
these ideals in a purely combinatorial way since, if $t$ is the (finite) index of $\delta(\Lambda[G])$ in $\zeta(\Lambda[G])$, then $\mathfrak{n}_{v}^a(L/K)$ and $\mathfrak{m}_{v}^a(L/K)$ both contain the lowest common multiple of the denominators of the coefficients of the element $t^{-1}(1-e_v+e_{\bf 1})e_{(a),S}$ of $\QQ[G]$. See \cite[Ex. 4.3 and 5.4]{dbalb} for examples of explicit computations.
}\end{remark}

\begin{remark} \label{exp rem}
{\em The observations made in Lemma \ref{fit vanishing} can be used to make the results of Theorem \ref{main conj} and of Remarks \ref{dualversionuseful} and \ref{exp cor} more explicit. 
%
If $\widehat G_{(a),S}=\widehat G$, then $e_{(a),S} = 1$ and so the
computation of $\mathfrak{n}^a_{v}(L/K)$ and $\mathfrak{m}_{v}^a(L/K)$ is easier. 
%
%
Assume now $S\not= S^a_{\rm min}$. Then Lemma \ref{fit vanishing} (iii) implies $(e_v-e_{\bf 1})e_{(a),S}{\rm Fit}_{\ZZ[G]}^a(\tilde h) = (0)$ for $v\in S\setminus S^a_{\rm min}$. Hence, for such $v$ the left hand side of (\ref{prediction1}) is  $e_{(a),S}\cdot (\Lambda\cdot{\rm Fit}_{\ZZ[G]}^a(\tilde h))$. 
 In particular, if both $\widehat G_{(a),S}=\widehat G$ and $S^a_{\rm min} \not= S$ then for $v\in S\setminus S^a_{\rm min}$, the left hand side of (\ref{prediction1}) is $\Lambda\cdot{\rm Fit}_{\ZZ[G]}^a(\tilde h)$ and $\mathfrak{n}^a_{v}(L/K) = \mathfrak{m}_{v}^a(L/K) = \xi(\Lambda[G])$.

}
\end{remark}

\subsection{The proof of Theorem \ref{main conj}}\label{6.2}

For notational ease, we only prove Theorem \ref{main conj} in the case $\Lambda=\ZZ$ (with the proof for $\Lambda=\ZZ_{(p)}$ being identical). Throughout \S \ref{6.2}, we fix $a< |S|$.

\subsubsection{} We first establish a useful fact concerning descent. 
%
For any normal subgroup $H$ of $G$ with $\Gamma:=G/H$ we set
$$\Theta_{L^H/K,S,T}^{a}:=\xi(\ZZ[\Gamma])\cdot \{\theta^{(a)}_{L^H/K,S,T}(0)\cdot R(\phi): \phi \in \Hom_\Gamma(\mathcal{O}_{L^H,S,T}^{\times},X_{L^H,S})\},$$
with $R(\phi)$ defined exactly as in (\ref{equivariantreg}) but for $L^H/K$, and regard it as a submodule of $\zeta(\CC[G])$ via the isomorphism $\zeta(\CC[\Gamma])\cong e_H\zeta(\CC[G])$.
We use the notation $N_I$ from Theorem \ref{main conj}(iv).



%
%
%

\begin{proposition}\label{last} Fix $I$ in $\wp_a^*(S)$. Fix a normal subgroup $H$ of $G$ that contains $N_I$ and set $\Gamma:=G/H$. 
Then $Y_{L^H,I}$ is a free $\ZZ[\Gamma]$-module of rank $a$ and $e_H(\Theta^{I}_{L/K,S,T})=\Theta_{L^H/K,S,T}^{a}.$
%

%
%
\end{proposition}

\begin{proof} 
It is clear that each place in $I$ splits completely in $L^H/K$ and hence that $Y_{L^H,I}$ is a free $\ZZ[\Gamma]$-module of rank $a$. To prove the remainng assertion, we note that $e_H\xi(\ZZ[G])=\xi(\ZZ[\Gamma])$ (by \cite[Lem. 3.2 (v)]{bses}), $e_{H}\cdot\theta^{(a)}_{L/K,S,T}(0)=\theta^{(a)}_{L^H/K,S,T}(0)$ and
, with $e_I^\Gamma$ the sum of primitive idempotents $e$ of $\zeta(\CC[\Gamma])$ with $e(\CC\cdot\ker\bigl(X_{L^H,S}\to Y_{L^H,I}\bigr))=0$, Lemma \ref{stick vanishing} implies  $\theta^{(a)}_{L^H/K,S,T}(0)=e_I^\Gamma\theta^{(a)}_{L^H/K,S,T}(0)$. It is thus enough to show that
\begin{equation}\label{enoughregs}e_H\cdot\{e_I R(\psi):\psi\in\Hom_G^{I}(\co_{L,S,T}^\times,X_{L,S})\}=\{e_I^\Gamma R(\phi): \phi \in \Hom_\Gamma(\mathcal{O}_{L^H,S,T}^{\times},X_{L^H,S})\}.\end{equation}
We fix a place $w_I$ in $(S\setminus I)_L$. For $v \in I$, we write $\kappa_v: \ZZ[G]\to X_{L,S}$ and $\kappa_v': \ZZ[\Gamma] \to X_{E,S}$ for the maps sending $g \in G$ to $g(w_v-w_I)$ and $\gamma\in \Gamma$ to $\gamma(w_{v,E}- w_{I,E})$. For $\varphi_\bullet=(\varphi_v)_{v\in I}$ in $\Hom_G(\co_{L,S,T}^\times,\ZZ[G])^{a}$, respectively $\Hom_\Gamma(\co_{L^H,S,T}^\times,\ZZ[\Gamma])^{a}$, we write $\psi_{\varphi_\bullet}$ for ${\sum}_{v \in I}\kappa^I_v\circ \varphi_v\in\Hom_G(\co_{L,S,T}^\times,X_{L,S})$, respectively $ {\sum}_{v \in I}\kappa'_v\circ \varphi_v \in \Hom_\Gamma(\mathcal{O}_{L^H,S,T}^{\times},X_{L^H,S})$. Then, since $\Hom_\Gamma^{I}(\co_{L^H,S,T}^\times,X_{L^H,S})=\Hom_\Gamma(\co_{L^H,S,T}^\times,X_{L^H,S})$ (cf. Remark \ref{easy split case}), the argument of Lemma \ref{ell inv interpretation} reduces the verification of (\ref{enoughregs}) to showing that
\begin{equation}\label{enoughregs2}e_H\cdot\{R(\psi_{\varphi_\bullet}):\varphi_\bullet\in\Hom_G(\co_{L,S,T}^\times,\ZZ[G])^{a}\}=\{R(\phi_{\varphi_\bullet}): \varphi_\bullet\in\Hom_\Gamma(\co_{L^H,S,T}^\times,\ZZ[\Gamma])^{a}\}.\end{equation}

We next write $\varrho_H$ for the natural composite homomorphism
\begin{equation*}\label{varrhoH}\Hom_{G}(\co_{L,S,T}^{\times},\ZZ[G])\to\Hom_{G}(\co_{L^H,S,T}^{\times},\ZZ[G])\to\Hom_{\Gamma}(\co_{L^H,S,T}^{\times},\ZZ[\Gamma]).\end{equation*} 
The first arrow here is surjective since $\mathcal{O}_{L,S,T}^{\times}/\mathcal{O}_{L^H,S,T}^{\times}$ is torsion-free. In addition, one has $\Hom_G(\mathcal{O}_{L^H,S,T}^{\times},\ZZ[G]) = \Hom_G(\mathcal{O}_{L^H,S,T}^{\times},T_H\cdot \ZZ[G])$ and the projection $\ZZ[G] \to \ZZ[\Gamma]$ sends $T_H$ to $|H|$. It follows that $\im(\varrho_H)=|H|\cdot \Hom_{\Gamma}(\mathcal{O}_{L^H,S,T}^{\times},\ZZ[\Gamma])$ and so (\ref{enoughregs2}) is obtained via 
\begin{align*}
&\,e_H\cdot\{R(\psi_{\varphi_\bullet}):\varphi_\bullet\in\Hom_G(\co_{L,S,T}^\times,\ZZ[G])^{a}\}\\=&\{{\rm Nrd}_{\CC[\Gamma]}(e_H(\CC\cdot\psi_{\varphi_\bullet})\circ e_H(R_{L,S})^{-1}):\varphi_\bullet\in\Hom_G(\co_{L,S,T}^\times,\ZZ[G])^{a}\}\\
=&\{{\rm Nrd}_{\CC[\Gamma]}(\CC\cdot\phi_{\varphi_\bullet}\circ e_H(R_{L,S})^{-1}):\varphi_\bullet\in|H|\cdot\Hom_{\Gamma}(\mathcal{O}_{L^H,S,T}^{\times},\ZZ[\Gamma])^{a}\}\\
=&\{{\rm Nrd}_{\CC[\Gamma]}(\CC\cdot\phi_{\varphi_\bullet}\circ (|H|e_H)(R_{L,S})^{-1}):\varphi_\bullet\in\Hom_{\Gamma}(\mathcal{O}_{L^H,S,T}^{\times},\ZZ[\Gamma])^{a}\}\\
=&\{{\rm Nrd}_{\CC[\Gamma]}(\CC\cdot\phi_{\varphi_\bullet}\circ R_{L^H,S}^{-1}):\varphi_\bullet\in\Hom_{\Gamma}(\mathcal{O}_{L^H,S,T}^{\times},\ZZ[\Gamma])^{a}\}\\
=&\{R(\phi_{\varphi_\bullet}): \varphi_\bullet\in\Hom_\Gamma(\co_{L^H,S,T}^\times,\ZZ[\Gamma])^{a}\}.&
\end{align*}
Here, the fourth equality uses the functoriality of the Dirichlet regulator map under change of field (as expressed by the commutativity of the diagram in \cite[Chap. I, \S6.5]{tate}).\end{proof}

\subsubsection{}We can now prove Theorem \ref{main conj}(i). At the outset we fix a place $v\in S$, a presentation $\tilde h=\tilde h_v$ of the $G$-module $\mathcal{S}_S^T(L)^{\rm tr}$ as in Theorem \ref{main result alg} and $a\in\NN_0$ with $a< |S|$. We set $\mathcal{F}_a:={\rm Fit}_{\ZZ[G]}^a(\tilde h)$.  
 %
%
To derive the decomposition (\ref{prediction1}) from Theorem \ref{main result alg} it is enough to prove 
\begin{equation}\label{first needed reduced} e_{\bf 1}\cdot \mathcal{F}_a =\begin{cases} 
0, &\text{if }a<|S|-1,\\ 
e_{\bf 1}\cdot \theta^{(a)}_{K/K,S,T}(0)\cdot R_{K,S,T}^{-1}\cdot\ZZ, &\text{if }a=|S|-1,\end{cases}\end{equation}
%
%
and, if ${\wp}_a(S,v)\not=\emptyset$, also
\begin{equation}\label{second needed} e_I\cdot  \mathcal{F}_a =  \Theta^{I}_{L/K,S,T} \,\,\,\text{ for all }\,\,I \in {\wp}_a(S,v).  \end{equation}
%
%

We then observe that (\ref{first needed reduced}) is valid if $a< |S|-1$ because $e_{\bf 1}\cdot \mathcal{F}_a=0$ by Lemma \ref{fit vanishing}(ii). 
Next we note that if $a = |S|-1$, then Remark \ref{asinLY}(i) with $H=G$ implies that $e_{\bf 1}\cdot \mathcal{F}_{a} = e_{\bf 1}\cdot {\rm Fit}_\ZZ^a(\mathcal{S}^T_S(K)^{\rm tr})$. 
%
Since $X_{K,S}$ is a free $\ZZ$-module of rank $a$, one also has 
\[ {\rm Fit}^{a}_{\ZZ}(\mathcal{S}_S^T(K)^{\rm tr})={\rm Fit}^{a}_{\ZZ}({\rm Cl}_S^T(K)\oplus X_{K,S})={\rm Fit}^0_{\ZZ}({\rm Cl}_S^T(K))=|{\rm Cl}_S^T(K)|\cdot\ZZ.\]
%
The verification of (\ref{first needed reduced}) now follows from the analytic class number formula for $K$. 

Turning now to (\ref{second needed}), we note first that if $G$ is trivial, then $\wp_a(S,v)\not=\emptyset$ only if both $a = |S|-1$ and $v = v_0$ in which case $\wp_a(S,v) = \{I\}$ with $I := S\setminus \{v_0\}$ and $e_I = e_{\bf 1} =  1$. In this case, therefore, the equality (\ref{second needed}) is equivalent to (\ref{first needed reduced}) and so is proved above.

In the sequel we assume $G$ is not trivial. In this case 
$e_I \cdot e_{N_I} = e_I$. This equality, combined with Proposition \ref{last} and Remark \ref{asinLY}(i) (both applied to $H=N_I$), 
implies that (\ref{second needed}) is valid if one has ${\rm Fit}_{\ZZ[G/N_I]}^a(\tilde h_{N_I}) = \Theta^{a}_{L^{I}/K,S,T}$, where $\tilde h_{N_I}$ denotes the $N_I$-coinvariance of $\tilde h$. 

To prove this we note that, since places in $I$ split completely in $L^I$, Proposition \ref{ltc2}(iii) implies ${\rm LTC}(L^{I}/K)$ follows from ${\rm LTC}(L/K;\ZZ[G]e_{(a),S})$. In addition, $\tilde h_{N_I}$ validates Theorem \ref{main result alg} relative to $L^{I}/K$ (cf. Remark \ref{asinLY}(i)) and so is related to Weil-Stark elements by the result of \cite[Th. 5.2(ii)]{bms1} for $L^I/K$, $S$, $\Pi=\ZZ[G/N_I]$ and $\pi:X_{L^{I},S}\to Y_{L^{I},I}$, which gives
$${\rm Fit}_{\ZZ[G/N_I]}^a(\tilde h_{N_I}) = \{({\wedge}_{i=1}^{i=a}\varphi_i)(\gamma_T\cdot\omega):\,\omega\in\W^{\Pi,\pi}_{L^{I}/K,S},\,  \varphi_i\in\Hom_{G/N_I}(\mathcal{O}_{L^{I},S,T}^{\times},\ZZ[G/N_I])\}.$$ 
The required equality follows directly upon combining the latter equality with Corollary \ref{Weil-Linvariants} and Example \ref{exam 2r}(i), 
also applied to $L^I/K$.
This proves Theorem \ref{main conj}(i). 

\subsubsection{}Turning now to claims (ii) and (iii) of Theorem \ref{main conj}, we first let $a$ be any integer with $0\leq a<|S|$. 
The equality (\ref{prediction1}) combines with (\ref{transpose equation}) and Remark \ref{asinLY}(ii) to imply that
\begin{align}\label{nasform}\mathfrak{m}_{v}^a(L/K)\cdot\iota_\#\bigl({{\sum}}_{I\in \wp_a(S,v)} \Theta_{L/K,S,T}^I\bigr)&\subseteq \mathfrak{m}_{v}^a(L/K)\cdot(1-e_v+e_{{\bf 1}})e_{(a),S}\cdot\iota_\#\bigl({\rm Fit}^{a}_{\ZZ[G]}(\tilde h)\bigr)\\
\notag&=\mathfrak{m}_{v}^a(L/K)\cdot(1-e_v+e_{{\bf 1}})e_{(a),S}\cdot{\rm Fit}^{{\rm tr},a}_{\ZZ[G]}(\tilde h^{\rm tr})\\
\notag&\subseteq \mathfrak{m}_{v}^a(L/K)\cdot(1-e_v+e_{{\bf 1}})e_{(a),S}\cdot{\rm Fit}^{{\rm tr},a}_{\ZZ[G]}(\mathcal{S}_S^T(L))
\end{align}
is contained in ${\rm Fit}^{{\rm tr},a}_{\ZZ[G]}(\mathcal{S}_S^T(L))$, and hence proves claim (ii). In order to prove claim (iii) we now assume that $a<|S|-1$, and fix $v$, $\varphi$ and $x$ as in this result. Then claim (iii) follows directly from (\ref{nasform}) and the fact that Lemma \ref{exam 2} (iii) implies 
\begin{align}\label{replaceLinv} (1-e_v)e_{(a),S}\cdot x\cdot \iota_\#(\mathscr{L}^a(\varphi))&=x\cdot \iota_\#\bigl({{\sum}}_{I\in \wp_a(S,v)}e_I(\theta^{(a)}_{L/K,S}\cdot R(\varphi))\bigr)\\  \notag&=\iota_\#(\gamma_T^{-1})\cdot x\cdot \iota_\#\bigl({{\sum}}_{I\in \wp_a(S,v)}e_I(\theta^{(a)}_{L/K,S,T}\cdot R(\varphi))\bigr)\\ \notag&\in\iota_\#(\gamma_T^{-1})\cdot\mathfrak{m}_{v}^a(L/K)\cdot\iota_\#\bigl({{\sum}}_{I\in \wp_a(S,v)} \Theta_{L/K,S,T}^I\bigr).\end{align}
%

\begin{remark}\label{seeremarkbelow}
{\em The containment $\mathfrak{n}_{v}^a(L/K)\cdot\iota_\#\bigl({\sum}_{I\in \wp_a(S,v)} \Theta_{L/K,S,T}^I\bigr)\subseteq\xi(\Lambda[G])$ claimed in Remark \ref{dualversionuseful} follows in just the same way after replacing $\mathfrak{m}_{v}^a(L/K)$ by $\mathfrak{n}_{v}^a(L/K)$ in (\ref{nasform}). 

Similarly, in the setting of Theorem \ref{main conj}(iii), if $x\in \mathfrak{n}^a_v(L/K)$ then this modified containment combines with (\ref{replaceLinv}) to imply the left-hand side of (\ref{Linvproduct}) belongs to $\iota_\#(\gamma_T^{-1})\cdot\xi(\Lambda[G])$.
}\end{remark}

\subsubsection{}
To prove Theorem \ref{main conj}(iv), we fix $a\in\NN_0$ with $a< |S|$, $I\in \wp_a^\ast(S)$ and $v_0\in S\setminus I$ and set $N:= N_I$ and $L^{I}:=L^{N}$. Then one has $$|N|\Theta_{L/K,S,T}^{I}=(|N|e_I)\Theta_{L/K,S,T}^{I}=(|N|e_{N}e_I)\Theta_{L/K,S,T}^{I}=T_{N}\cdot\Theta^{I}_{L/K,S,T}$$ and so it is enough to show that 
$\delta(\ZZ[G])\cdot T_{N}\cdot\Theta^{I}_{L/K,S,T}\subseteq{\rm Ann}_{\ZZ[G]}({\rm Cl}_{S'}^T(L))$.

Writing $q_{I}:\CC[G]\to\CC[G/N_I]$ for the canonical map, 
we next observe that
\begin{equation}\label{Iprojection}q_I\bigl(\delta(\ZZ[G])\cdot \Theta^{I}_{L/K,S,T}\bigr)\subseteq\delta(\ZZ[G/N])\cdot \Theta^{a}_{L^{I}/K,S,T}.\end{equation} Here we have combined Proposition \ref{last} with \cite[Lem. 3.5 (vii)]{bses}.

We now combine the result of \cite[Th. 5.2(iii), Rem. 5.4]{bms1} for $L^I/K$, $S$, $\Pi=\ZZ[G/N]$ and $\pi:X_{L^{I},S}\to Y_{L^{I},I}$, which gives that
$$\delta(\ZZ[G/N])\cdot\{({\wedge}_{i=1}^{i=a}\varphi_i)(\gamma_T\cdot\omega):\,\omega\in\W^{\Pi,\pi}_{L^{I}/K,S},\,  \varphi_i\in\Hom_{G/N}(\mathcal{O}_{L^{I},S,T}^{\times},\ZZ[G/N])\}$$ is contained in ${\rm Ann}_{\ZZ[G/N]}({\rm Cl}_{S'}^T(L^{I}))$, 
with Corollary \ref{Weil-Linvariants} and Example \ref{exam 2r}(i), and deduce that $\delta(\ZZ[G/N])\cdot \Theta^{a}_{L^{I}/K,S,T}\subseteq{\rm Ann}_{\ZZ[G/N]}({\rm Cl}_{S'}^T(L^{I}))$.

In view of (\ref{Iprojection}), the necessary inclusion $\delta(\ZZ[G])\cdot T_{N}\cdot\Theta^{I}_{L/K,S,T}\subseteq{\rm Ann}_{\ZZ[G]}({\rm Cl}_{S'}^T(L))$ now follows because for a given element $y$ of $\CC[G]$, one has $T_{N}\cdot y\in \ZZ[G]$ if and only if $q_I(y)\in \ZZ[G/N]$, and the action of $T_{N}$ induces the field-theoretic norm map ${\rm Cl}^T_{S'}(L)\to{\rm Cl}^T_{S'}(L^{I}).$
This completes the proof of Theorem \ref{main conj}. 

\section{Special cases}\label{last section} 

In this section we specialise to cases in which the leading term conjecture is known to be valid and so Theorem \ref{main conj} is  unconditional. In particular, we prove Theorem \ref{end result intro} and \ref{lastone2}. 

\subsection{Totally real fields}\label{TR section} In this section we assume $L$ is totally real and fix an odd prime $p$. At present, the main evidence in support of the leading term conjecture in this case is provided by the following result. In this result we write $\mu_p(L)$ for the $p$-adic cyclotomic $\mu$-invariant of $L$ and we refer to `Breuning's local epsilon constant conjecture' formulated in \cite{breuning} and to the `$p$-adic Stark Conjecture at
$s=1$' that is discussed by Tate in \cite[Chap. VI, \S5]{tate}, where it is attributed to Serre \cite{serre1}.

\begin{proposition}\label{iwasawa theory+} If $L$ is totally real, then ${\rm LTC}(L/K,\ZZ_{(p)}[G])$ is valid whenever  $p$ is odd and satisfies all of the following conditions.
\begin{itemize}
\item[(i)] $p$ is prime to $|G|$ or $\mu_p(L)$ vanishes;
\item[(ii)] Breuning's local epsilon constant conjecture is valid for all extensions obtained by $p$-adically completing $L/K$;
\item[(iii)] The $p$-adic Stark Conjecture at $s=1$ is valid for all $\CC_p$-valued characters of $G$. 
\end{itemize}
\end{proposition}

\begin{proof} With $S$ and $T$ as in (\ref{data fix}), we define an element of  $K_0(\ZZ[G],\RR[G])$ by setting $T\Omega:= \delta_{G}(\theta^*_{L/K,S,T}(0))-\chi_{\ZZ[G]}(C_{L,S,T},R_{L,S})$.  
%
%
Then, under condition (iii), the argument of Johnston and Nickel in \cite[Th. 8.1(i), Prop. 7.1]{JLMS} implies that $T\Omega$ belongs to $K_0(\ZZ[G],\QQ[G])$. 
In this case, by Proposition \ref{ltc2}, ${\rm LTC}(L/K,\ZZ_{(p)}[G])$ is thus equivalent to asserting  $T\Omega$ belongs to the kernel of the scalar extension map $K_0(\ZZ[G],\QQ[G]) \to K_0(\ZZ_{(p)}[G],\QQ[G])$. 

In particular, since the scalar extension map $K_0(\ZZ_{(p)}[G],\QQ[G])\to K_0(\ZZ_p[G],\QQ_p[G])$ is bijective, both groups being naturally isomorphic to the Grothendieck group of finite $G$-modules that have $p$-power order and finite projective dimension (cf. \cite[Vol. II, p.73]{curtisr}), it is therefore enough to show that, under the given hypotheses, the image $T\Omega_p$ of $T\Omega$ under the scalar extension map $K_0(\ZZ[G],\QQ[G]) \to K_0(\ZZ_p[G],\QQ_p[G])$ vanishes. 

Next we recall that, by the argument given in \cite[\S9.1]{dals}, one knows that $T\Omega_{p}$ vanishes provided that all of the following conditions are satisfied: the $p$-adic Stark Conjecture at
$s=1$ is valid for all $\CC_p$-valued characters of $G$; if $p$ divides $|G|$, then $\mu_p(L)$ vanishes; the `$p$-component' of a certain element $T\Omega^{\rm loc}(\QQ(0)_L,\ZZ[G])$ of $K_0(\ZZ[G],\RR[G])$ vanishes.

The element $T\Omega_p$ therefore does vanish under the given hypotheses, since the result \cite[Th. 4.1]{breuning} of Breuning implies that the $p$-component of $T\Omega^{\rm loc}(\QQ(0)_L,\ZZ[G])$ vanishes provided that the `local epsilon constant conjecture' formulated in \cite{breuning} is valid for all extensions that are obtained by completing $L/K$ at a $p$-adic place.
%
%
\end{proof}

%


\begin{remark}\label{muinvariants}{\em If $p$ divides $|G|$, then the assumption $\mu_p(L)=0$ guarantees that the main conjecture of non-commutative Iwasawa theory is valid for $L^{\rm cyc}/K$, with $L^{\rm cyc}$ the cyclotomic $\ZZ_p$-extension of $L$. However, in \cite{JN3}, Johnston and Nickel identify families of $L/K$ for which one can prove this conjecture without assuming $\mu_p(L)=0$ (or $p\nmid|G|$). In such cases, Proposition \ref{iwasawa theory+} implies ${\rm LTC}(L/K,\ZZ_{(p)}[G])$ whenever $p$ satisfies (ii) and (iii).  More generally, as explained in \cite[Rem. 4.3]{JN3}, $\mu_p(L)=0$ if the classical Iwasawa `$\mu=0$' conjecture is valid for $L(\zeta_p)$ at $p$, where $\zeta_p$ denotes a primitive $p$-th root of unity. Thus a result of Ferrero and Washington \cite{fw} implies that $\mu_p(L)=0$ if $L$ is an abelian field.
}\end{remark}

\begin{remark}\label{bleycobbe}{\em Breuning's conjecture is valid for all tamely ramified extensions of local fields \cite[Th. 3.6]{breuning} and for certain classes of wildly ramified extensions (cf. Bley and Cobbe \cite{bc}). 
}\end{remark}

\begin{remark}\label{padicStark}{\em The $p$-adic Stark Conjecture at $s=1$ for the trivial character of $G$ is equivalent to Leopoldt's conjecture for $K$ and $p$. Moreover, in \cite{JLMS}, Johnston and Nickel identify families of $L/K$ for which ${\rm LTC}(L/K,\ZZ_{(p)}[G])$ is valid if $p$ satisfies conditions (i) and (ii) of Proposition \ref{iwasawa theory+} as well as Leopoldt's conjecture for $L$ and $p$. They also prove that the $p$-adic Stark Conjecture at $s=1$ is valid for all absolutely abelian characters.
}\end{remark}



To end this section, we offer a concrete example of the sort of consequences that follow by combining Proposition \ref{iwasawa theory+} 
with Theorem \ref{main conj}.

\begin{corollary}\label{tr intro} Assume $L$ is a totally real ramified extension of $K$. Set $S := S_K^\infty \cup S^{\rm ram}_{L/K}$ and fix  a set $T$ as in (\ref{data fix}) and a prime $p$ as in Proposition \ref{iwasawa theory+}. 
%
%
%
%
Then for $\alpha\in \delta(\ZZ_{(p)}[G])$,  $v\in S^{\rm ram}_{L/K}$, $g\in G_v$ and 
 $\psi\in\Hom_G(\co_{L,S,T}^\times,X_{L,S})$, the product $\alpha\cdot(g-1)\cdot\theta^{([K:\QQ])}_{L/K,S,T}(0)\cdot R(\psi)$ belongs to $\ZZ_{(p)}[G]$ and annihilates ${\rm Cl}^T(L)_{(p)}$. 
\end{corollary}

\begin{proof} In this case, 
%
%
the first observation in Remark \ref{easy split case} applies with $I = S_K^\infty$ to show that $\Hom_G(\co_{L,S,T}^\times,X_{L,S})= \Hom^I_G(\co_{L,S,T}^\times,X_{L,S})$. In addition, for this $I$, the group $N_I$ is trivial and it can be easily checked that $e_I\cdot \theta^{([K:\QQ])}_{L/K,S,T}(0) = \theta^{([K:\QQ])}_{L/K,S,T}(0)$. Finally, we note that the observations in Remark \ref{dals improve cor} can also be applied with $V = S_K^\infty$ (so $a=[K:\QQ]$) and $v \in S_{L/K}^{\rm ram}$. 

Taken all of these observations into account, we can now apply Theorem \ref{main conj}(iv) to $L/K$ with this choice $I$ in order to deduce that  $\alpha\cdot \theta^{([K:\QQ])}_{L/K,S,T}(0)\cdot R(\psi)$ 
%
%
belongs to $\ZZ_{(p)}[G]$ and annihilates ${\rm Cl}^T_{S'}(L)_{(p)}$ with $S' = S_K^\infty\cup \{v\}$. 
 The claimed result is therefore true since, if $g\in G_v$, then $g-1$ annihilates the quotient of ${\rm Cl}^T(L)_{(p)}$ by ${\rm Cl}^T_{S'}(L)_{(p)}$. 
\end{proof}

\subsection{CM fields} 

In this section we assume $L$ is CM and $K$ is totally real. We fix an odd prime $p$ and an isomorphism $\CC\cong\CC_p$ (the choice of which will not matter in the sequel) and thereby identify $\G$ with the set of irreducible $\CC_p$-valued characters of $G$. We write $\tau$ for the element of $G$ induced by complex conjugation, $e_-$ for the central idempotent $(1-\tau)/2$ of $\QQ[G]$ and $\G^{-}$ for the subset of $\G$ comprising characters for which one has $\chi(\tau) = -\chi(1)$. For any $G$-module $M$ we also write $M^-$ for the $G$-submodule $\{m \in M: \tau(m) = - m\}$.

\subsubsection{} 
For each place $w$ of $L$, Gross has defined in \cite[\S1]{G0} a local $p$-adic absolute value 
\[ ||\cdot ||_{w,p}: L_w^\times \xrightarrow{r_w} G_{L_w^{\rm ab}/L_w} \xrightarrow{\chi^{-1}_w} \ZZ_p^\times,\]
where $L_w^{\rm ab}$ denotes the maximal abelian extension of $L_w$ (in a fixed algebraic closure of $L_w$), $r_w$ the reciprocity map of local class field theory and $\chi_w$ the $p$-adic cyclotomic character.

For an abelian group $M$ we set $M_p:=\ZZ_p\otimes_\ZZ M$. For any finite set of places $\Sigma$ of $K$ that contains both $S^\infty_K$ and the set of all $p$-adic places of $K$, we write $R_{L,\Sigma}^p$ for the homomorphism of $\bz_p[G]$-modules $\mathcal{O}^{\times,-}_{L,\Sigma,p} \to Y^-_{L,\Sigma,p}$ that sends each $u$ in $\mathcal{O}^{\times,-}_{L,\Sigma}$ to ${\sum}_{\pi\in \Sigma_L}\mathrm{log}_p||u||_{\pi,p} \cdot \pi$.

We write $\G^{\rm ss}$ for the subset of $\G^-$ comprising characters $\rho$ for which the homomorphism
$\Hom_{\CC_p[G]}(V_{\check\rho},\CC_p\cdot R_{L,\Sigma}^p)$ is injective (for a $\CC_p$-representation $V_{\check\rho}$ of character $\check\rho$). We obtain an idempotent (that is independent of the set $\Sigma$) by setting 
$e_{\rm ss} := {\sum}_{\rho\in \G^{\rm ss}}e_\rho\in\zeta(\QQ_p[G])$.
%

\begin{remark}\label{GSC remark}{\em In \cite[Conj. 1.15]{G0} Gross conjectures that each $\rho\in\G^-$ belongs to $\G^{\rm ss}$, and hence that $e_{\rm ss}=e_-$. The first of these predictions is commonly referred to as the Gross-Kuz'min Conjecture for $\rho$ (since it is related to earlier work of Kuz'min \cite{kuzmin}). We further recall that, by \cite[Th. 3.1(i) and (iii)]{burns2}, $\rho$ has this property if and only if it satisfies the `Order of Vanishing conjecture' formulated by Gross in \cite[2.12a)]{G0}.}\end{remark}


\subsubsection{} We now review the main evidence for the leading term conjecture in this setting.
 
\begin{proposition}\label{cm evidence} ${\rm LTC}(L/K,\ZZ_{(p)}[G]e_{\rm ss})$ is valid if $p$ is odd and satisfies the following:
\begin{itemize}
\item[(i)] either $p$ is prime to $|G|$ or $\mu_p(L)$ vanishes, and
\item[(ii)] the Weak $p$-adic Gross-Stark Conjecture (of \cite[Conj. 2.12b)]{G0}) is valid for $L/K$.
\end{itemize}
\end{proposition}

\begin{proof} We assume (without further comment) that conditions (i) and (ii) are satisfied. For any ring $R$ of characteristic zero we also set $R[G]^- := R[G]e_-$ and $R[G]^{\rm ss} := R[G]e_{\rm ss}$.  

By a Theorem of Siegel \cite{Siegel}, Stark's conjecture is valid for characters in $\G^-$. The element 
$$T\Omega(L/K)^-:=\delta_{\ZZ[G]^-,\RR[G]^-}(\theta^*_{L/K,S,T}(0)e_-)-\chi_{\ZZ[G]^-}(C_{L,S,T}^-,R_{L,S}^-)$$
therefore belongs to  $K_0(\ZZ[G]^-,\QQ[G]^-) \subset K_0(\ZZ[G]^-,\RR[G]^-)$ (cf. \cite[Rem. 6.1.1(ii)]{dals}). 

It is thus enough to show $T\Omega(L/K)^-$ belongs to the kernel of the composite scalar extension map 
\[ K_0(\ZZ[G]^-,\QQ[G]^-) \xrightarrow{\alpha_p} K_0(\ZZ_{(p)}[G]^-,\QQ[G]^-)\xrightarrow{\beta_p} K_0(\ZZ_{(p)}[G]^{\rm ss},\QQ[G]^{\rm ss}).\]
 In addition, the argument of \cite[Cor. 3.8(i)]{burns2} implies   
$\alpha_p(T\Omega(L/K)^-)$ belongs to the image of $(\zeta(\QQ[G])(e_- -e_{\rm ss}))^\times$ under the map $\delta^-:= \delta_{\ZZ_{(p)}[G]^-,\QQ[G]^-}$. It is thus enough to note that, since $\ZZ_{(p)}[G]^-$ is semilocal, Bass's Theorem implies that $\delta^-\bigl((\zeta(\QQ[G])(e_- -e_{\rm ss}))^\times\bigr) = \ker(\beta_p)$ (cf. the discussion concerning the diagram \cite[(37)]{burns2}).
%
%
%
%
%
%
\end{proof} 

\begin{corollary}\label{pminuspart} The conjecture ${\rm LTC}(L/K,\ZZ_{(p)}[G]e_{-})$ is valid if $p$ is odd and satisfies either one of the following conditions:
\begin{itemize}\item[(i)] $\mu_p(L)$ vanishes and the Gross-Kuz'min Conjecture is valid for $L/K$;
\item[(ii)] the Sylow $p$-subgroups of $G$ are abelian.
\end{itemize}
\end{corollary}
\begin{proof} In case (i) one has $e_{\rm ss}=e_-$ and, upon combining the results of \cite[Th. 2.6 and Th. 3.1(i),(iii)]{burns2}, one deduces that the validity of the Gross-Kuz'min Conjecture for characters in $\G^-$ implies the validity of the Weak $p$-adic Gross-Stark Conjecture for characters in $\widehat{G}^-$. The validity of ${\rm LTC}(L/K,\ZZ_{(p)}[G]e_{-})$ in case (i) thus follows from Proposition \ref{cm evidence}.

We now consider case (ii) and the element $T\Omega(L/K)^-$ of $K_0(\ZZ[G]^-,\QQ[G]^-)$ defined above. By the same argument as in the proof of Proposition \ref{iwasawa theory+}, it is enough to show that the projection $T\Omega(L/K)^-_p$ of $T\Omega(L/K)^-$ to $K_0(\ZZ_p[G]^-_p,\QQ_p[G]^-_p)$ vanishes.

Now, upon combining a result \cite[Th. 1]{NickelSS} of Nickel with \cite[Rem. 6.1.1(iii)]{dals}, one knows that $T\Omega(L/K)^-_p$ has finite order. It then follows from \cite[Prop. 6.2]{NickelIntegrality} that $T\Omega(L/K)^-_p$ vanishes if and only $T\Omega(L'/K')^-_p$ vanishes for every intermediate Galois CM extension $L'/K'$ of $L/K$ whose Galois group is either $p$-elementary or a direct product of a $p$-elementary group with $\{1,\tau\}$. In addition, by the argument of \cite[\S 3]{NickelSS}, the assumption that the Sylow $p$-subgroups of $G$ are abelian implies that every $p$-elementary subquotient of $G$ is abelian.

To complete the proof, it is therefore enough to show that $T\Omega(L'/K')^-_p$ vanishes for every intermediate abelian extension $L'/K'$ of $L/K$ in which $L'$ is CM and $K'$ totally real. For such extensions $L'/K'$, however, the validity of ${\rm LTC}(L'/K',\ZZ_p[{\rm Gal}(L'/K')]e_-)$, and hence the vanishing of $T\Omega(L'/K')^-_p$, has been derived by Bullach, Daoud and the first and third named authors \cite[Th. B (a)]{bbds} from the seminal work of Dasgupta and Kakde in \cite{dk}.
\end{proof}

\begin{remark}\label{Nickel6.8}{\em Nickel has also proved in \cite[Th. 6.8]{NickelIntegrality} that if $\mu_p(L)$ vanishes and $p$ is `non-exceptional', in the sense of \cite[Def. 6.5]{NickelIntegrality}, then ${\rm LTC}(L/K,\ZZ_{(p)}[G]e_{-})$ is valid. We recall that there can only be finitely many `exceptional' primes.
}\end{remark}

\subsubsection{}\label{cm intro deduction} 
%
%

As a concrete application, we now combine Corollary \ref{pminuspart}(ii) with Theorem \ref{main conj}(iii) in order to derive Theorem \ref{end result intro}(ii). In fact, strictly speaking, we need a slight extension of Theorem \ref{main conj}(iii) . Specifically, for any idempotent $e_{-}$ as in Example \ref{exam1} (iii), the argument of Theorem \ref{main conj} shows that if ${\rm LTC}(L/K;\Lambda[G](e_{-}\cdot e_{(a),S}))$ is valid, then claims (ii), (iii) and (iv) remain valid after multiplying each product on their left-hand sides by $e_{-}$. 

In particular, for $e_-=(1-\tau)/2$, the validity of ${\rm LTC}(L/K,\ZZ_{(p)}[G]e_{-})$ in the setting of Theorem \ref{end result intro}(ii) implies that, for each $v\in S$ and $x\in\mathfrak{m}_{v}^a(L/K)$, one has 
\begin{equation}\label{almost there}e_-(1-e_v)e_{(a),S}\cdot \iota_\#(x)\cdot \mathscr{L}^a(|G|\varphi)\,\in \,\gamma_T^{-1}\cdot\iota_\#\bigl(\ZZ_{(p)}\otimes_{\ZZ}{\rm Fit}^{{\rm tr},a}_{\ZZ[G]}(\mathcal{S}_S^T(L))\bigr)\end{equation}

We now set $e_a:=e_- e_{(a),S}$. Then, by fixing $v$ in $S_K^\infty$,  and noting that we are assuming $a<|S|-1$, Lemma \ref{fit vanishing}(ii) implies that
$$(1-e_v+e_{\bf 1})e_{(a),S}\bigl(\ZZ_{(p)}\otimes_{\ZZ}{\rm Fit}^{{\rm tr},a}_{\ZZ[G]}(\mathcal{S}_S^T(L))\bigr)=e_a\bigl(\ZZ_{(p)}\otimes_{\ZZ}{\rm Fit}^{{\rm tr},a}_{\ZZ[G]}(\mathcal{S}_S^T(L))\bigr).$$
In this case, we can therefore take $x$ to be $p^t$ for any $t\in\NN_0$ with $p^te_a\in \ZZ_{(p)}[G]$.

With these choices of $v$ and $x$ one then finds that the left-hand side of (\ref{almost there}) is equal to
$$e_- e_- e_{(a),S}\cdot \iota_\#(p^t)\cdot \mathscr{L}^a(|G|\varphi)=p^te_a\cdot \mathscr{L}^a(|G|\varphi),$$
as required to complete the proof of Theorem \ref{end result intro}(ii) (with the curent choice of $e_a$).


To finally justify the claim made immediately after the statement of Theorem \ref{end result intro}, we note that for $a=0$ one has $\widehat{G}_{(0),S} = \widehat{G}$ so $e_{(0),S}=1$, $e_0=e_-$ and  $t=0$. Therefore, with $\varphi':= |G|\varphi$, one has 
$$p^0e_0\cdot \mathscr{L}^0(\varphi')=e_-\theta_{L/K,S}^{\langle0\rangle}(0)R(\varphi')=e_-\theta_{L/K,S}^{\langle0\rangle}(0)e_{0,S}R(\varphi')=e_-\theta_{L/K,S}^{\langle0\rangle}(0)=\theta_{L/K,S}^{\langle0\rangle}(0),$$
where the third equality holds because $e_{0,S}(\RR\cdot X_{L,S})$ vanishes.

\begin{remark}{\em One can formulate a `$p$-adic' analogue of ${\rm LTC}(L/K,\ZZ_{(p)}[G]e_{\rm ss})$ in which the roles of Artin $L$-series and Dirichlet regulators are replaced by $p$-adic Artin $L$-series (as discussed by Greenberg in \cite{greenberg}) and Gross's $p$-adic regulators. The argument of \cite[Cor. 3.8(i)]{burns2} implies this analogous conjecture is valid modulo the vanishing of $\mu_p(L)$ (and even in some cases unconditionally). Such results can be combined with the methods developed here to prove analogues of Theorem \ref{end result intro} in this setting. 
}\end{remark}

\subsection{Function fields}\label{gffs} We assume in this section that $K$ is a global function field and we recall that ${\rm LTC}(L/K)$ has been verified by Kakde and the first author in \cite{bk}. 

In this case, therefore, the arithmetic properties of Artin $\mathscr{L}$-invariants derived in \S \ref{6.1} are unconditionally valid.  In particular, just as in \S \ref{cm intro deduction}, it is straightforward to derive case (i) of Theorem \ref{end result intro} as a consequence of Theorem \ref{main conj}(iii) after setting $e_a:=(1-e_v)e_{(a),S}$ for an arbitrary place $v\in S$. We omit the details but, instead, we explicitly show that Theorem \ref{main conj} implies a higher-order, non-abelian generalisation of a result of Deligne.


To do this we define the $T$-modified degree zero divisor class group ${\rm Pic}^{T,0}(L)$ of $L$ to be the cokernel of the natural divisor map from $L_T^\times$ to the group of divisors of $L$ of degree zero with support outside $T$.
This map is injective (as $L_T^\times$ is torsion-free since $T$ is non-empty) and the $G$-module ${\rm Pic}^{T,0}(L)$ is an extension of ${\rm Pic}^{0}(L)$ by a finite group.

\begin{corollary}\label{gff cor}Let $K$ be a global function field and fix $a\in\NN_0$ with $a<|S|$. Then, for $v\in S$, one has $\mathfrak{m}_{v}^a(L/K)\cdot\iota_\#\bigl({\sum}_{I\in \wp_a(S,v)} \Theta_{L/K,S,T}^I\bigr)\subseteq{\rm Fit}^{{\rm tr},a}_{\ZZ[G]}({\rm Pic}^{T,0}(L)^\vee).$ 

If in addition $a<|S|-1$ then for $v\in S$, $\varphi\in \bigcap_{I\in \wp_a(S,v)}\Hom_G^{I}(\co_{L,S,T}^\times,X_{L,S})$ and $x\in\mathfrak{m}_{v}^a(L/K)$, one has
\begin{equation*}(1-e_v)e_{(a),S}\cdot x\cdot \iota_\#(\gamma_T\cdot\mathscr{L}^a(\varphi))\,\in \,{\rm Fit}^{{\rm tr},a}_{\ZZ[G]}({\rm Pic}^{T,0}(L)^\vee).\end{equation*}
In particular, this containment is valid for every $\varphi$ in $|G|\cdot\Hom_G(\co_{L,S,T}^\times,X_{L,S})$.
\end{corollary}

\begin{proof} Following Theorem \ref{main conj}(ii) and (iii), we are reduced to showing that ${\rm Fit}_{\ZZ[G]}^{{\rm tr},a}(\mathcal{S}^T_{S}(L)) \subseteq {\rm Fit}_{\ZZ[G]}^{{\rm tr},a}({\rm Pic}^{T,0}(L)^\vee)$. To do this it is in turn enough, by \cite[Th. 3.17(vi)]{bses}, to prove that ${\rm Pic}^{T,0}(L)^\vee$ is isomorphic to a quotient of $\mathcal{S}^T_{S}(L)$.

We write $\theta_T$ for the homomorphism defined as in (\ref{definingSelmer}) but with $S$ taken to be empty. Then there exists a natural surjection from $\mathcal{S}^T_{S}(L)$ to ${\rm cok}(\theta_T)$ and so it suffices to prove that the latter group is naturally isomorphic to ${\rm Pic}^{T,0}(L)^\vee$.

We write $\Delta_T$ for the (injective) divisor map from $L_T^\times$ to the group ${\rm Div}_T(L)$ of divisors of $L$ with support outside $T$ and note ${\rm Pic}^{T,0}(L)$ identifies with the (finite) torsion subgroup of ${\rm cok}(\Delta_T)$. It is then enough to note the functor $\Hom_\ZZ(-,\ZZ)$ applies to the tautological short exact sequence $0 \to L_T^\times \to {\rm Div}_T(L) \to {\rm cok}(\Delta_T)\to 0$ to give an exact sequence
\[ {\prod}_{w \notin T_L}\ZZ \xrightarrow{\theta_T} \Hom_\ZZ(L_T^\times,\ZZ) \to {\rm Ext}^1_\ZZ({\rm cok}(\Delta_T),\ZZ) \to 0\]
and ${\rm Ext}^1_\ZZ({\rm cok}(\Delta_T),\ZZ) ={\rm Ext}^1_\ZZ({\rm cok}(\Delta_T)_{\rm tor},\ZZ)\cong {\rm Ext}^1_\ZZ({\rm Pic}^{T,0}(L),\ZZ) = {\rm Pic}^{T,0}(L)^\vee$.
\end{proof}

\begin{remark}{\em If $\widehat G_{(a),S} = \widehat G$ and $S^a_{\rm min}\not=S$, then Remark \ref{exp rem} implies  $\mathfrak{m}^a_{S,T}(L/K)= \xi(\ZZ[G])$ contains $1$ and so can be omitted from the inclusion in Corollary \ref{gff cor}. 
 }\end{remark}

\begin{remark}\label{deligne rem}{\em Corollary \ref{gff cor} specialises to recover the Brumer-Stark Conjecture for $L/K$, as proved by Deligne (cf. \cite[Chap. V]{tate}). To see this note that if $a=0$, then $S^a_{\rm min}=\emptyset$ and Lemma \ref{fit vanishing} (iii) implies  $(1-e_v-e_{\bf 1})e_{(0),S}\cdot{\rm Fit}_{\ZZ[G]}^{{\rm tr},0}(\mathcal{S}^T_S(L)) = {\rm Fit}_{\ZZ[G]}^{{\rm tr},0}(\mathcal{S}^T_S(L))$ for each $v\in S$. This implies that 
$\mathfrak{m}^0_{v}(L/K) = \xi(\Z[G])$ and hence $1\in \mathfrak{m}^0_{v}(L/K)$ and so Corollary \ref{gff cor} implies $\Theta_{L/K,S,T}^\emptyset\subseteq\iota_\# ({\rm Fit}^{{\rm tr},0}_{\ZZ[G]}({\rm Pic}^{T,0}(L)^\vee)).$ 
Hence, for any $\psi\in\Hom_G(\co_{L,S,T}^\times,X_{L,S})$ the element $\theta_{L/K,S,T}(0)=e_\emptyset\theta_{L/K,S,T}(0)=e_\emptyset\theta_{L/K,S,T}^{(0)}(0)R(\psi)$ belongs to $\iota_\#\bigl({\rm Fit}^{{\rm tr},0}_{\ZZ[G]}({\rm Pic}^{T,0}(L)^\vee)\bigr).$
%
This in turn implies Deligne's result since, if $G$ is abelian, then $\iota_\#\bigl({\rm Fit}^{{\rm tr},0}_{\ZZ[G]}({\rm Pic}^{T,0}(L)^\vee)\bigr)$ is contained in $\iota_\#\left({\rm Ann}_{\ZZ[G]}({\rm Pic}^{T,0}(L)^\vee)\right)={\rm Ann}_{\ZZ[G]}({\rm Pic}^{T,0}(L)).$
}\end{remark}


\subsection{Chinburg-Stark elements}\label{cs}

If $\psi\in \widehat{G}$ satisfies the Strong-Stark conjecture, then as explained in Remark \ref{Picomponent}, the claim of Theorem \ref{eTNC} is valid in the setting of Remark \ref{stark context}. This observation combines with the specialisation of \cite[Th. 5.2]{bms1} to strengthen the refined Stark conjecture formulated in \cite[Conj. 2.6.1]{dals}. 
By means of a concrete application, we shall now use this approach to prove Theorem \ref{lastone2}. 

We use the notation of Remark \ref{stark context}. In particular, we identify each $\psi$ in $\widehat{G}$ with a representation $G \to {\rm GL}_{\psi(1)}(\mathcal{O}_\psi)$, abbreviate $\co_\psi$ to $\co$ and write $E$ for the field of fractions of $\co$. We write ${\rm tr}_{E/\QQ}$ for the canonical trace map. We also abbreviate the functors $M \mapsto {^{\Pi_\psi}}M$ and $M \to {_{\Pi_\psi}}M$ to $M \to M^\psi$ and $M \to M_\psi$ respectively,
and set ${\rm pr}_\psi:={\sum}_{g\in G}\psi(g)g^{-1}$.


\begin{proposition}\label{very last} Fix $\theta$ in $\CC\cdot\Hom_{G}(\mathcal{O}^\times_{L,S,T},X_{L,S})$. Then, for each non-trivial character $\psi$ in $\widehat{G}$ that satisfies the Strong-Stark Conjecture, one has 
\begin{equation*} |G|\cdot{\rm tr}_{E/\QQ}( L_{S,T}^{*}(\check\psi,0){\rm det}_\CC(|G|m\cdot (R_{L,S}^{-1}\circ \theta)^\psi)\cdot{\rm pr}_\psi) \in
{\rm Ann}_{\ZZ[G]}({\rm Cl}^T_{S'}(L)).\end{equation*}
Here $m$ is any element of 
$\mathcal{O}_\psi$ with $m\cdot\theta(\mathcal{O}_{L,S,T}^{\times,\psi})\subseteq (X_{L,S,\psi})_{\rm tf}$ and $S'$ is any subset of $S$ as described in \cite[Th. 5.2(iii)]{bms1} with respect to $\pi = \pi^\psi$.
\end{proposition}

\begin{proof} We set $U:= \mathcal{O}^\times_{L,S,T}$ and $X := X_{L,S}$. 
Since the given hypotheses imply that the claim of Theorem \ref{eTNC} is unconditionally valid for $\Pi_\psi$, we shall deduce the claimed result by applying the containment of \cite[Th. 5.2(iii)]{bms1} to the homomorphism $\pi=\pi^\psi$ and a suitable Stark element as in Remark \ref{stark context}.
We set $\gamma_{T,\psi}:=\iota_{\Pi_\psi}(\gamma_{T})$.

It suffices to prove the displayed containment after $p$-localisation. To do this we fix a prime $p$ and a subset $\underline{b} = \{b_i\}_{1\le i\le r}$ of $(X_\psi)_{\rm tf}$ which gives an $\mathcal{O}_{(p)}$-basis of $(X_\psi)_{{\rm tf},(p)}$. For each integer $i$ with $1\leq i\leq r$ we set $\theta_{b_i}:= b_i^*\circ \theta$. Then for any element $n$ of $\mathcal{O}$ one has
\begin{equation}\label{very last equality} L_{S,T}^{*}(\check\psi,0)\cdot{\rm det}_\CC(n\cdot (R_{L,S}^{-1}\circ \theta)^\psi) = (\wedge_{i=1}^{i=r}(n\cdot\theta_{b_i}))(\gamma_{T,\psi}\cdot\varepsilon^\pi_{\underline{b}}).\end{equation}

%
%
%
%
%

Now, if $m\cdot\theta$ maps the sublattice $U^\psi$ of $E\cdot U$ to $X_{\psi,{\rm tf}}$ then 
we may apply \cite[Th. 5.2(iii)]{bms1} with
$\varphi_i:=|G|m\cdot \theta_{b_i}\in \Hom_{\mathcal{O}}(U^\psi,\mathcal{O})_{(p)}$ for each index $i$, with $a=|G|\in\co=\delta(\mathcal{O})$ and with $\epsilon$ the trace map ${\rm tr}_{E/\QQ}$. 
The claimed result now follows directly by combining this result with the equality (\ref{very last equality}) for $n = |G|m$ and with Theorem \ref{eTNC}. 
\end{proof}


\begin{remark}\label{very last remark}{\em The containment in Proposition \ref{very last} is finer than the prediction made in \cite[Conj. 2.6.1]{dals} in that the term $\psi(1)^{-1}|G|^{3+r}e_\psi = |G|^{2+r}{\rm pr}_\psi$ and group ${\rm Cl}_{S'}(L)$ that occur in loc. cit. are replaced by $|G|^{1+r}{\rm pr}_\psi$ and ${\rm Cl}^T_{S'}(L)$. We also recall that, if $p$ is odd, then Nickel \cite[Th. 1]{NickelSS} has recently proved the $p$-part of the Strong Stark Conjecture for totally odd characters of totally real fields.
}\end{remark}

Turning now to the proof of Theorem \ref{lastone2}, we note that the first three claims are proved in \cite[Th. 4.3.1(ii) and Prop. 12.2.1]{dals}. The annihilation statement of Theorem \ref{lastone2} (iv) is however finer than that of \cite[Prop. 12.2.1]{dals}. The key point in its proof is that the hypotheses on $\psi$ and $S$ that are made in Theorem \ref{lastone2} imply ${\rm dim}_{E}(E\cdot X_{L,S^\infty_K,\psi}) = 1$ and hence that the set $S' = S^\infty_K$ satisfies the hypothesis of \cite[Th. 5.2(iii)]{bms1} (and therefore also Proposition \ref{very last}) with respect to the homomorphism $\pi = \pi^\psi$.

Given the observation in Remark \ref{very last remark}, to prove Theorem \ref{lastone2}(iv) one need only make the following two changes to the proof of \cite[Prop. 12.2.1(iv)]{dals}: the use of the containment \cite[(44)]{dals} is replaced by the stronger containment discussed in Remark \ref{finer ssc rem} below (with $r = 1$ and $S' = S^\infty_K$) and the use of \cite[Lem. 11.1.2(i)]{dals} is replaced by the finer computation used in the proof of \cite[Th. 5.2(iii)]{bms1}. 

\begin{remark}\label{finer ssc rem}{\em In the notation used in the proof of Proposition \ref{very last}, the $\mathcal{O}_{(p)}$-modules $U_{(p)}^\psi$ and $(X_{\psi})_{{\rm tf},(p)}$ are both free of rank $r$, so one can choose $\theta$ with $\theta(U_{(p)}^\psi) = (X_{\psi})_{{\rm tf},(p)}$. For such $\theta$, the equality (\ref{very last equality}) with $n=1$ combines with \cite[Prop. 5.11]{bms1} to imply that
%
%
\[ |G|^{1+r}(\wedge_{i=1}^{i=r}\theta_{b_i})(\gamma_{T,\psi}\cdot\varepsilon_{\underline{b}}) = |G| L_{S,T}^{*}(\check\psi,0)\cdot{\rm det}_\CC(|G|\cdot (R_{L,S}^{-1}\circ \theta)^\psi) \in {\rm Ann}_{\mathcal{O}} ({\rm Cl}_{S'}^T(L)_\psi)_{(p)}.\]
This implies, by choice of $\theta$, that $|G|^{1+r}\gamma_{T,\psi}\cdot\varepsilon_{\underline{b}} \in {\rm Ann}_{\mathcal{O}} ({\rm Cl}_{S'}^T(L)_\psi)\cdot ({\wedge}^{r}_{\mathcal{O}}U^\psi)_{(p)}$ and hence  
\begin{align*}\label{finer ssc}|G|^{1+r}L_{S,T}^{*}(\check\psi,0)\cdot ({\wedge}^r_{\mathcal{O}}X_{\psi,{\rm tf}})_{(p)}
&=\, \mathcal{O}_{(p)}\cdot |G|^{1+r}L_{S,T}^{*}(\check\psi,0)\wedge_{i=1}^{i=r}b_i\\
&=\, \mathcal{O}_{(p)}\cdot (\wedge_{\CC}^r R_{L,S})(|G|^{1+r}\gamma_{T,\psi}\cdot\varepsilon_{\underline{b}})\notag\\
&\subseteq\, {\rm Ann}_{\mathcal{O}} ({\rm Cl}_{S'}^T(L)_\psi)\cdot ({\wedge}^{r}_{\CC}{\rm R}_{L,S})({\wedge}^{r}_{\mathcal{O}}U^\psi)_{(p)}.\end{align*}
Since this is true for all $p$ it refines the containment \cite[(44)]{dals} in which ${\rm Cl}_{S'}(L)_\psi$ rather than
${\rm Cl}_{S'}^T(L)_\psi$ occurs (and, in this regard, we also recall Remark \ref{dals improve}.)   
}\end{remark}

\end{document}